\documentclass[preprint, 12pt, sort&compress]{elsarticle}

\usepackage{latexsym}
\usepackage{amsfonts}
\usepackage{graphicx,amssymb,
}
\usepackage{amsmath}
\usepackage{amssymb}
\usepackage[mathscr]{eucal}
\usepackage{enumerate}
\usepackage{amsthm}

\begin{document}

\newtheorem{tm}{Theorem}[section]
\newtheorem{pp}{Proposition}[section]
\newtheorem{lm}{Lemma}[section]
\newtheorem{df}{Definition}[section]
\newtheorem{tl}{Corollary}[section]
\newtheorem{re}{Remark}[section]
\newtheorem{eap}{Example}[section]

\newcommand{\pof}{\noindent {\bf Proof} }
\newcommand{\ep}{$\quad \Box$}

\newcommand{\al}{\alpha}
\newcommand{\be}{\beta}
\newcommand{\var}{\varepsilon}
\newcommand{\la}{\lambda}
\newcommand{\de}{\delta}
\newcommand{\st}{\stackrel}

\allowdisplaybreaks


\begin{frontmatter}

\title{Characterizations of compact sets in fuzzy sets spaces with $L_p$ metric
\tnoteref{t1}
 }
\tnotetext[t1]{Project supported by
 National Natural Science
Foundation of China (No. 61103052), and by Natural Science Foundation of Fujian Province of China(No. 2016J01022)}
\author[huang]{Huan Huang \corref{cor1}}
\cortext[cor1]{Corresponding author}
\address[huang]{Department of Mathematics, Jimei
University, Xiamen 361021, China}
 \ead{hhuangjy@126.com }

\author[wu]{Congxin Wu}
 \address[wu]{Department of
Mathematics, Harbin Institute of Technology, Harbin 150001, China}
\ead{wucongxin@hit.edu.cn}

\date{}

\begin{abstract}
%
Compactness criteria for fuzzy sets spaces endowed with $L_p$ metric
have been studied for several decades.
In metric spaces, totally boundedness   is  a key feature  of compactness.
However,
compare the existing compactness criteria for fuzzy sets spaces endowed with $L_p$ metric with Arzel\`{a}--Ascoli theorem,
we can see that
the latter
gives the compactness criteria by characterizing the totally bounded sets
 while
the former does not seem to characterize the totally bounded sets.
Besides, till now,
 compactness criteria are only presented for
three particular fuzzy sets spaces,
they both have assumptions of convexity  or star-shapedness.
In recent years, general fuzzy sets become more and more important in both theory and applications.
Motivated by needs listed above,
in this paper,
 we
present characterizations of totally bounded sets, relatively compact sets
 and
compact sets in the fuzzy sets spaces $F_B(\mathbb{R}^m)$ and $F_B(\mathbb{R}^m)^p$ equipped with $L_p$ metric,
where
 $F_B(\mathbb{R}^m)$
 and
$F_B(\mathbb{R}^m)^p$
are two kinds of general fuzzy sets on $ \mathbb{R}^m$ which do not have any assumptions of convexity  or star-shapedness.
 Subsets
 of
 $F_B(\mathbb{R}^m)^p$   include common fuzzy sets
such as fuzzy numbers, fuzzy star-shaped numbers with respect to the origin,
 fuzzy star-shaped numbers,
and the general fuzzy star-shaped numbers introduced by Qiu et al.
The existed compactness criteria
are
stated for three kinds of fuzzy sets spaces endowed with $L_p$ metric
whose universe sets are the former three kinds of common fuzzy sets respectively.
 Constructing
completions
of   fuzzy sets spaces  with respect to $L_p$ metric is a problem which is closely dependent on characterizing totally bounded sets.
 Based on preceding characterizations of totally boundedness and relatively compactness
and some discussions on   convexity and star-shapedness of fuzzy sets,
we
show
 that the completions of
 fuzzy sets spaces mentioned in this paper can be obtained by using the $L_p$-extension.
We also clarify
 relation among all the ten fuzzy sets spaces discussed in this paper,
which consist of five pairs of original spaces and the corresponding completions.
Then,
we
show that the subspaces of $F_B(\mathbb{R}^m)$
 and
$F_B(\mathbb{R}^m)^p$ mentioned in this paper have parallel
 characterizations of totally bounded sets, relatively compact sets
 and
compact sets.
At last,
as applications of our results, we discuss properties of
$L_p$ metric on fuzzy sets space
and
 relook compactness criteria proposed in previous work.

\end{abstract}


\end{frontmatter}


\section{Introduction}

Compactness is a fundamental property in both theory and applications \cite{kelley,fs,gh3}.
The research of compactness criteria   attracts much attention.
 It's well-known that  Arzel\`{a}--Ascoli theorem(s)
provide compactness criteria in classic analysis and topology.
Fuzzy systems
have been successfully
used
to
 solve
many real-world problems \cite{aliev, he, chen, zeng, chanussot}.
Undoubtedly, compactness plays an important role in analysis and applications of fuzzy sets and systems \cite{bu,roman2,huang3,huang6,huang8,wa,zeng}.
There exist
many important and interesting works including \cite{da2,da3,greco,greco3,huang,ma,wu2,zhao,fan,roman,roman2}
which   characterized   compactness
in fuzzy sets spaces equipped with different topologies.

Since Diamond and Kloeden \cite{da3} introduced $d_p$ metric which is a $L_p$-type metric,
it has become one of the most often used
convergence structure
on fuzzy sets.
Naturally, people have started to consider characterizations of compactness in
fuzzy sets spaces with $d_p$ metric.

Diamond and Kloeden \cite{da3}
gave   compactness criteria for   fuzzy number space $E^m$
with
  $d_p$ metric.
Ma \cite{ma} modified the characterization given by   Diamond and Kloeden.
Convexity
is
 a very useful property. Star-shapedness is    a natural extension of convexity.
Of
cause,
research on
fuzzy counterparts of star-shaped sets
 has aroused the interest of people \cite{da, chanussot}.
Diamond \cite{da2}
introduced the fuzzy star-shaped numbers as an extension of fuzzy numbers.
$S^m$
is
used to denote the set of all fuzzy star-shaped numbers.
Diamond \cite{da2}
characterized the compact sets in
$(S^m_0, d_p)$, where $S^m_0$ denotes the set of all the fuzzy star-shaped numbers with respect to the origin.
$E^m$ and $S^m_0$ do not include each other.
They both are subsets of $S^m$.

Wu
and Zhao \cite{wu2} pointed out that the compactness criteria for $(E^m, d_p)$ and $(S^m_0, d_p)$   in \cite{da2, da3}
have
the same type of error
 and
that
 the modified compactness criteria in \cite{ma} still has fault.
They    \cite{wu2}
gave
right
 characterizations of compactness
in
$(S^m_0, d_p)$ and $(E^m, d_p)$.
Based on the results in \cite{wu2},
Zhao and Wu \cite{zhao}
further
proposed
a characterization of compactness
in
$(S^m, d_p)$.
In these discussions,
 it is found that
 the concepts
``$p$-mean equi-left-continuous'' and ``uniformly $p$-mean bounded''
proposed by  Diamond and Kloeden \cite{da3} and Ma \cite{ma}, respectively,
play an important role
in establishing and illustrating
 characterizations of compactness in
fuzzy sets spaces with $d_p$ metric.

Compare the characterizations in \cite{wu2, zhao} to Arzel\`{a}--Ascoli theorem,
we find that
the latter
provides the compactness criteria by characterizing the totally bounded sets
 while the former does not seem to characterize the totally bounded sets.
Since, in metric space, totally boundedness   is  a key feature  of compactness,
it is a
natural and important problem to consider how to characterize totally bounded sets
in fuzzy sets spaces with $d_p$ metric?

The existing three compactness criteria in \cite{wu2,zhao} are stated for
three
particular fuzzy sets spaces with $d_p$ metric.
Fuzzy sets
in
these
spaces have assumption of convexity  or star-shapedness.
It is worth noting
that
general fuzzy sets, which have no assumptions of convexity  or star-shapedness, have attracted more and more attention
in
both theory and applications \cite{ba, kloeden, zeng}.
So
this has caused a basic and important problem:
how
to
characterize totally bounded sets, relatively compact set
and
compact
sets in   general fuzzy sets spaces endowed with $d_p$ metric?

There is another motivation to think about the problem above.
Clearly,
there exists other types of particular fuzzy sets.
For example,
Qiu et al.\cite{qiu}
introduced
the set of
all
general fuzzy star-shaped numbers,
which is denoted by $\widetilde{S}^m$.
$S^m$ is a subset
of
$\widetilde{S}^m$ which in turn is a subset of $F_{B}(  \mathbb{R}^m )$.
If we characterize totally bounded sets, relatively compact sets and compact sets
in
general fuzzy sets spaces,
then
we
can
obtain   parallel    characterizations for particular fuzzy sets spaces immediately because the latter are subspaces of the former.

Analysis of Diamond \cite{da}   indicates
that $(E^m, d_p)$, $(S^m_0, d_p)$, $(S^m, d_p)$ and $(\widetilde{S}^m, d_p)$
are
not
complete.    Kr\"{a}tschmer \cite{kratschmer} presented the completion of $(E^m, d_p)$
which is described by
the support functions of fuzzy numbers.
It
is
 natural to consider a basic
problem:
what are the completions of all the rest spaces?
Perhaps this question should be replaced by a more general question: how to generate the completions of fuzzy sets
spaces
with respect to $d_p$ metric?
This
problems are closely relevant to the problem of characterizing
 totally bounded sets and relatively compact sets
in fuzzy sets spaces equipped
with
$d_p$ metric.

In this paper, we want
to answer all questions above.
These questions
 are closely relevant to each other.
It can even be said that they are different aspects of a same problem.
To
put our discussion in a more general setting
which
 does not have any assumptions of convexity  or star-shapedness,
we
consider $F_{B}(  \mathbb{R}^m )$
which is
the set of all
 normal, upper semi-continuous, compact-support     fuzzy sets on $ \mathbb{R}^m$.
Further
we
introduce the $L_p$-extension of a fuzzy sets space.
The $L_p$-extensions of $(S_0^{m}, d_p)$,
$(E^m, d_p)$, $(S^m, d_p)$, $(\widetilde{S}^m, d_p)$ and $(F_{B}(  \mathbb{R}^m ), d_p)$ are denoted by $(S_0^{m,p}, d_p)$,
$(E^{m,p}, d_p)$, $(S^{m,p}, d_p)$, $(\widetilde{S}^{m,p}, d_p)$
and
$(F_{B}   (  \mathbb{R}^m )^p, d_p) $, respectively.
All the fuzzy sets spaces
mentioned
in this paper are subspaces of $(F_{B}   (  \mathbb{R}^m )^p,    d_p) $.
We give
characterizations
of totally bounded sets,
 relatively compact sets, and compact sets
in
$(F_{B}   (  \mathbb{R}^m ), d_p) $
and
$(F_{B}   (  \mathbb{R}^m )^p,    d_p) $.
Then
it is proved that
each
$L_p$-extension mentioned in this paper is exactly the completion of its original space.
Next,
we
show
that
the subspaces of $(F_{B}   (  \mathbb{R}^m ), d_p) $
and
$(F_{B}   (  \mathbb{R}^m )^p,    d_p) $
have parallel characterizations
of
totally bounded sets,
 relatively compact sets, and compact sets
to
that of them.
Finally,
as applications of our results, we consider properties of
$d_p$ metric
and
 relook characterizations of compactness proposed in previous work.

The remainder part of this paper is organized as follows.
Since $d_p$ metric is based on the well-known Hausdorff metric,
Section 2
 introduces
and
discusses
some properties of Hausdorff metric.
In
Section 3, we recall and introduce some concepts and results of fuzzy sets related to our paper.
Then,
in Section 4,
we present characterizations of relatively compact sets, totally bounded sets and compact sets
in
$(F_{B}   (  \mathbb{R}^m )^p ,      d_p)$.
Section 5
shows
that
 $F_{B}   (  \mathbb{R}^m )^p $
 is
 in fact the completion
 of
 $F_{B}   (  \mathbb{R}^m ) $ according to $d_p$ metric.
Then we
give
 characterizations of relatively compact sets, totally bounded sets and compact sets
in
$(F_{B}   (  \mathbb{R}^m ) ,      d_p)$.
Based on the conclusions in Sections 4 and 5 and some discussions on convexity and star-shapedness of fuzzy sets, in Section 6,
we
show that
$(S_0^{m,p}, d_p)$,
$(E^{m,p}, d_p)$, $(S^{m,p}, d_p)$ and $(\widetilde{S}^{m,p}, d_p)$
are
exactly
the completions
of
$(S_0^{m}, d_p)$,
$(E^m, d_p)$, $(S^m, d_p)$ and $(\widetilde{S}^m, d_p)$,
respectively.
We
clarify
relation among the ten fuzzy sets spaces discussed in this paper.
As consequences of preceding results,
it
follows characterizations of totally bounded sets, relatively compact sets and compact sets
in these spaces.
  Section 7 gives some applications of the results in our paper.
At last, we draw conclusions in Section 8.

\section{The Hausdorff metric}

Let $\mathbb{N}$ be the set of all natural numbers,
$\mathbb{Q}$
 be the set of all rational numbers,
$\mathbb{R}^m$
be $m$-dimensional Euclidean space,
$K_C(\mathbb{R}^m)$
be
the set of
all the nonempty compact and convex sets in $\mathbb{R}^m$,
$K(  \mathbb{R}^m  )$
be
the
set of all nonempty compact set in $\mathbb{R}^m$,
and
$C(  \mathbb{R}^m  )$
be
the
set of all nonempty closed set in $\mathbb{R}^m$.
The
 well-known
 {\rm Hausdorff} metric $H$ on
   $C(\mathbb{R}^m)$ is defined by:
$$H(U,V)=\max\{H^{*}(U,V),\ H^{*}(V,U)\}$$
for arbitrary $U,V\in C(\mathbb{R}^m)$, where
  $$H^{*}(U,V)=\sup\limits_{u\in U}\,d\, (u,V) =\sup\limits_{u\in U}\inf\limits_{v\in
V}d\, (u,v).$$

\begin{pp} \cite{da} \label{kcs}
 $(C(\mathbb{R}^m), H)$ is a complete metric space in which $K(\mathbb{R}^m)$ and $K_C(  \mathbb{R}^m  )$
are closed subsets.
Hence, $K(\mathbb{R}^m)$ and $K_C(\mathbb{R}^m)$ are also complete metric spaces.
\end{pp}

\begin{pp}\cite{da, roman} \label{sca}
A
nonempty subset $U$ of $(   K(\mathbb{R}^m),    H     )$
is
compact
if and only if
it is closed and bounded in $(   K(\mathbb{R}^m),    H     )$.
\end{pp}

\begin{pp}\cite{da} \label{mce}
  Let $\{u_n\} \subset K(\mathbb{R}^m)$ satisfy
$ u_1 \supseteq u_2 \supseteq \ldots \supseteq u_n \supseteq \ldots.  $
Then
$u= \bigcap_{n=1}^{+\infty}  u_n \in K(\mathbb{R}^m)$ and
$H(u_n, u) \to 0  \ \hbox{as} \ n\to \infty$.

On the other hand, if $ u_1 \subseteq u_2 \subseteq \ldots \subseteq u_n \subseteq \ldots  $
and
$u= \overline{\bigcup_{n=1}^{+\infty}  u_n }\in K(\mathbb{R}^m)$,
then
$H(u_n, u) \to 0  \ \hbox{as} \ n\to \infty$.
\end{pp}

A set $K\in K(\mathbb{R}^m)$ is said to be star-shaped relative to a point $x\in K$
if
for each $y\in K$, the line $\overline{xy}$ joining $x$ to $y$
is contained in $K$.
The kernel ker\;$K$ of $K$ is the set of all points $x\in K$
such that
$\overline{xy}\subset K$
for each $y\in K$.
The symbol $K_S(\mathbb{R}^m)$
is used to
denote
all the star-shaped sets in $\mathbb{R}^m$.

Obviously,
$K_C(\mathbb{R}^m)\subsetneq K_S(\mathbb{R}^m)$.
It can be checked that
$\hbox{ker}\; K \in K_C(\mathbb{R}^m) $ for all $K\in K_S(\mathbb{R}^m)$.

We say that a sequence of sets $\{C_n\}$ converges to
$C$, in the sense of Kuratowski, if
$$
C
=
\liminf_{n\rightarrow \infty} C_{n}
=
\limsup_{n\rightarrow \infty} C_{n},
$$
where
\begin{gather*}
\liminf_{n\rightarrow \infty} C_{n}
 =
 \{x\in X: \  x=\lim\limits_{n\rightarrow \infty}x_{n},    x_{n}\in C_{n}\},
\\
\limsup_{n\rightarrow \infty} C_{n}
=
\{
 x\in X : \
 x=\lim\limits_{j\rightarrow \infty}x_{n_{j}},x_{n_{j}}\in C_{n_j}
\}
 =
 \bigcap\limits_{n=1}^{\infty}   \overline{   \bigcup\limits_{m\geq n}C_{m}    }.
\end{gather*}
In this case, we'll write simply
$C=\lim_{n\to\infty}C_n  (K)$.

The following two known propositions discuss the relation of the convergence induced by Hausdorff metric
and
the convergence
 in the sense of Kuratowski. The readers can see \cite{greco2} for details.

\begin{pp} \label{hms}
  Suppose that $u$, $u_n$, $n=1,2,\ldots$, are nonempty compact sets
  in $\mathbb{R}^m$.
  Then
  $H(u_n, u) \to 0$ as $n\to \infty$
  implies that
    $u=\lim_{n\to\infty}u_n  (K)$.
\end{pp}

\begin{pp}  \label{klhe}
  Suppose that $u$, $u_n$, $n=1,2,\ldots$, are nonempty compact sets
  in $\mathbb{R}^m$
  and
  that $u_n$, $n=1,2,\ldots$, are connected sets.
  If
    $u=\lim_{n\to\infty}u_n (K)$,
  then
   $H(u_n, u) \to 0$ as $n\to \infty$.
\end{pp}

\begin{tm} \label{ksc}
$K_S(\mathbb{R}^m)$ is a closed set in $( K(\mathbb{R}^m),   H )$.
\end{tm}

\begin{proof} \ Suppose that $\{ u_n \} \subset K_S(\mathbb{R}^m) $,  $u\in K(\mathbb{R}^m)$
and
 $H(u_n, u) \to 0$ as $n\to \infty$.
In the following,
we will prove that $u\in K_S(\mathbb{R}^m)$.

Choose $x_n\in {\rm Ker}\; u_n$, $n=1,2,\ldots$,
then there exists an $N$ such that
$x_n \in U$   for all $n\geq N$, where $U:=\{y:  \ d(y,u)\leq 1 \}$.
Note that $U$ is a compact set,
we
know
that
there is a subsequence $\{x_{n_i}\}$ of $\{x_n\}$ such that
$\lim_{i\to \infty }  x_{n_i} = x_0  $.
So
$x_0 \in \limsup_{n\to \infty} u_n$,
it then follows from Proposition \ref{hms}
that
$x_0\in u$.

Now, we show that $u$ is star-shaped  and    $x_0\in {\rm ker}\; u $.
It suffices to show that
$$\lambda x_0 +  (1-\lambda) z \in u$$
for all $z\in u$ and $\lambda\in [0,1]$.
In fact,
given $z\in u$, since $u=\liminf_{n\to \infty} u_n$,
there is a sequence $\{z_n: z_n\in u_n\}$ such that
$\lim_{n\to \infty}  z_n =z $.
Hence, for each $\lambda\in [0,1]$,
$$
\lambda x_0 +  (1-\lambda) z = \lim_{i\to \infty} \lambda x_{n_i} + (1-\lambda) z_{n_i}  \in \limsup_{n\to \infty} u_{n},
  $$
and thus, by Proposition \ref{hms}, $\lambda x_0 +  (1-\lambda) z \in u$. 
\end{proof}

\begin{tl} \label{kef}
  Let $u$, $u_n$ be star-shaped sets, $n=1,2,\ldots$.
 If $H(u_n, u) \to 0$,
then
$\limsup_{n\to \infty}\mbox{ker}\; u_n \subset \mbox{ker}\; u$.
\end{tl}

\begin{proof}
\ From the proof of Theorem \ref{ksc}, we get the desired results.
\end{proof}

\begin{re}
{\rm
  We do not know whether Theorem \ref{ksc} and Corollary \ref{kef} are known conclusions,
so
we give our proofs here.
}
\end{re}

\section{The spaces of fuzzy sets} \label{fsp}

In
this section,
we
recall and introduce
various
spaces of fuzzy sets including
fuzzy numbers space, fuzzy star-shaped numbers space and
general fuzzy star-shaped numbers space.
Some basic properties of these spaces are discussed.

We use $F(\mathbb{R}^m)$ to represent all
fuzzy subsets on $\mathbb{R}^m$, i.e. functions from $\mathbb{R}^m$
to $[0,1]$.
For details, we refer the readers to references
\cite{wu, da}.
$
\mathbf{2}^{\mathbb{R}^m }:=  \{  S: S\subseteq \mathbb{R}^m      \}
$
can be embedded in $F (\mathbb{R}^m)$, as any $S \subset \mathbb{R}^m$ can be
seen as its characteristic function, i.e. the fuzzy set
\[
\widehat{S}(x)=\left\{
\begin{array}{ll}
1,x\in S, \\
0,x\notin S.
\end{array}
\right.
\]
For
$u\in F(\mathbb{R}^m)$, let $[u]_{\al}$ denote the $\al$-cut of
$u$, i.e.
\[
[u]_{\al}=\begin{cases}
\{x\in \mathbb{R}^m : u(x)\geq \al \}, & \ \al\in(0,1],
\\
{\rm supp}\, u=\overline{\{x \in \mathbb{R}^m: u(x)>0\}}, & \ \al=0.
\end{cases}
\]
For
$u\in F(\mathbb{R}^m)$,
we suppose that
\\
(\romannumeral1) \ $u$ is normal: there exists at least one $x_{0}\in \mathbb{R}^m$
with $u(x_{0})=1$;
\\
(\romannumeral2) \ $u$ is upper semi-continuous;
\\
(\romannumeral3-1) \ $u$ is fuzzy convex: $u(\la x+(1-\la)y)\geq {\rm min} \{u(x),u(y)\}$
for $x,y \in \mathbb{R}^m$ and $\la \in [0,1];$
\\
(\romannumeral3-2$^0$) \ $u$ is fuzzy star-shaped with respect to the origin, i.e.,
$u(\lambda y )  \geq  u(y)$
for all $y\in \mathbb{R}^m$ and $\lambda\in [0,1]$.
\\
(\romannumeral3-2) \ $u$ is fuzzy star-shaped, i.e.,   there exists $x\in \mathbb{R}^m$
such that
$u$ is fuzzy star-shaped with respect to $x$, namely,
$u(\lambda y + (1-\lambda) x)  \geq  u(y)$
for all $y\in \mathbb{R}^m$ and $\lambda\in [0,1]$;
\\
(\romannumeral3-3) \ Given $\alpha\in (0,1]$, then there exists
$x_\alpha \in [u]_\alpha$
such that
$\overline{x_\alpha\, y} \in [u]_\alpha$
for all $y\in [u]_\alpha$;
\\
(\romannumeral4-1) \ $[u]_0$ is a bounded set in $\mathbb{R}^m$;
\\
(\romannumeral4-2) \ $\left (\int_0^1   H( [u]_\al, \{0\}   )^p  \; d\alpha   \right)^{1/p}  <   +\infty $,
where $p\geq 1$ and $0$ denotes the origin of $\mathbb{R}^m$;
\\
(\romannumeral4-3) \  $[u]_\al$ is a bounded set in $\mathbb{R}^m$ when $\al>0$.
\begin{itemize}
  \item
  If
     $u$ satisfies (\romannumeral1), (\romannumeral2), (\romannumeral3-1) and (\romannumeral4-1),
  then
   $u$ is a fuzzy number. The set of all fuzzy numbers is denoted by $E^m$.

 \item
  If $u$ satisfies (\romannumeral1), (\romannumeral2), (\romannumeral3-2$^0$) and (\romannumeral4-1),
  then
  we say $u$ is a fuzzy star-shaped number with respect to the   origin.
The
set of all fuzzy star-shaped  numbers with respect to the origin is denoted by
  $S_0^m$.

  \item
  If $u$ satisfies (\romannumeral1), (\romannumeral2), (\romannumeral3-2) and (\romannumeral4-1),
  then
  we say $u$ is a fuzzy  star-shaped number. The set of all fuzzy star-shaped  numbers is denoted by
  $S^m$.

\item
If $u$ satisfies (\romannumeral1), (\romannumeral2), (\romannumeral3-3) and (\romannumeral4-1),
  then
  we say $u$ is a  general fuzzy star-shaped number. The set of all  general fuzzy  star-shaped  numbers is denoted by
  $\widetilde{S}^m$.
\end{itemize}
The definitions of fuzzy star-shaped numbers and general fuzzy  star-shaped  numbers were given by Diamond \cite{da}
and Qiu et al. \cite{qiu}, respectively.
$\mathbb R^m$ can be embedded in $E^m$, as any $r \in  \mathbb{R}^m$ can be
viewed as the fuzzy number
\[
\widehat{r}(x)=\left\{
\begin{array}{ll}
1,x=r, \\
0,x\not=r.
\end{array}
\right.
\]
We can see that $E^m \nsubseteq S^m_0$
and
$S^m_0 \nsubseteq E^m$.
If $u\in S^m$, then $\bigcap_{\alpha\in (0,1]}\mbox{ker}\;[u]_\alpha \not= \emptyset$,
however
this inequality may not hold when $u\in \widetilde{S}^m$.
Clearly $S^m_0= \{u\in S^m: \ 0\in \bigcap_{\alpha\in (0,1]}\mbox{ker}\;[u]_\alpha\}$.
So
$E^m, S^m_0   \subsetneq   S^m   \subsetneq    \widetilde{S}^m$.

In order to illustrate and prove the conclusions in this paper,
we need to
introduce $L_p$--type noncompact fuzzy sets.

Let     $u\in  F(\mathbb{R}^m)$.
\begin{itemize}
  \item  If $u$ satisfies (\romannumeral1), (\romannumeral2), (\romannumeral3-1) and (\romannumeral4-2),
then
  we say $u$ is a $L_p$-type noncompact fuzzy number. The collection of all such fuzzy sets is denoted by $E^{m,p}$.

 \item
  If $u$ satisfies (\romannumeral1), (\romannumeral2), (\romannumeral3-2$^0$) and (\romannumeral4-2),
  then
  we say $u$ is a $L_p$-type noncompact fuzzy star-shaped number with respect to the   origin.
The collection of all such fuzzy sets is denoted by
  $S_0^{m,p}$.

  \item  If
  $u$ satisfies (\romannumeral1), (\romannumeral2), (\romannumeral3-2) and (\romannumeral4-2),
then
  we say $u$ is a $L_p$-type noncompact fuzzy  star-shaped number. The collection of all such fuzzy  sets is denoted by $S^{m,p}$.

  \item If
  $u$ satisfies (\romannumeral1), (\romannumeral2), (\romannumeral3-3) and (\romannumeral4-2),
then
  we say $u$ is a $L_p$-type noncompact general fuzzy  star-shaped number. The collection of all such fuzzy sets is denoted by $\widetilde{S}^{m,p}$.
\end{itemize}
It's easy to check that
$
E^m \subsetneq    E^{m,p}
$,
$
  S_0^m \subsetneq    S_0^{m,p}
$,
$
  S^m \subsetneq    S^{m,p}
$ and
$\widetilde{S}^{m} \subsetneq    \widetilde{S}^{m,p}$.

We
can see a kind of $L_p$--type noncompact fuzzy sets is obtained
by using weaker assumption (\romannumeral4-2) to replace
 stronger assumption (\romannumeral4-1) on the corresponding kind of compact fuzzy sets.
So
the latter is a subset of the former.
We call this process \emph{$L_p$-extension}.
A
kind of $L_p$--type noncompact fuzzy sets
is
called \emph{$L_p$-extension} of the corresponding kind of compact fuzzy sets.

Above eight kinds of fuzzy sets have assumptions of convexity  or star-shapedness.
In recent years,
people pay more and more attention to general fuzzy sets from points of view of   theoretical research and real-world applications. For example,
in
the study of fuzzy differential equations \cite{ba,kloeden}, researchers consider fuzzy sets with no assumptions of convexity or star-shapedness.
 For this reason, we wish
 discussions in this paper can be put in a setting of general fuzzy sets which have no assumption of convexity  or star-shapedness.
So
we
introduce
the following kinds of general fuzzy sets.

Suppose    $u\in  F(\mathbb{R}^m)$.
\begin{itemize}
  \item    If $u$ satisfies (\romannumeral1), (\romannumeral2) and (\romannumeral4-1),
then
 $u$ is a normal upper semi-continuous compact-support     fuzzy set on $ \mathbb{R}^m$. The collection of all such fuzzy sets is denoted by $F_{B}(  \mathbb{R}^m )$.

\item     If $u$ satisfies (\romannumeral1), (\romannumeral2) and (\romannumeral4-2),
then
$u$ is a  normal upper semi-continuous
$L_p$-type noncompact-support fuzzy set on $ \mathbb{R}^m$.
The collection of all such fuzzy sets is denoted by $F_{B}   (  \mathbb{R}^m )^p $.

 \item    If $u$ satisfies (\romannumeral1), (\romannumeral2) and (\romannumeral4-3),
then
 $u$ is a normal upper semi-continuous  noncompact-support     fuzzy set on $ \mathbb{R}^m$. The collection of all such fuzzy sets is denoted by $F_{GB}(  \mathbb{R}^m )$.
\end{itemize}
Clearly,
$F_{B}   (  \mathbb{R}^m )^p$
is
the $L_p$-extension of
$F_{B}   (  \mathbb{R}^m )$.
We can check that
\begin{gather*}
   E^{m},  S_0^{m} \subsetneq S^{m} \subsetneq \widetilde{S}^{m} \subsetneq   F_{B}(  \mathbb{R}^m ),
\\
    E^{m,p},  S_0^{m,p} \subsetneq S^{m,p} \subsetneq \widetilde{S}^{m,p} \subsetneq   F_{B}(  \mathbb{R}^m )^p,
\\
F_{B}(  \mathbb{R}^m ) \subsetneq  F_{B}   (  \mathbb{R}^m )^p  \subsetneq  F_{GB}(  \mathbb{R}^m ).
\end{gather*}

Diamond and Kloeden \cite{da} introduced
the
 $d_p$ distance ($1\leq p<\infty$) on $S^m$
which is defined by
\begin{equation}\label{dpm}
d_p\, (u,v)=\left(     \int_0^1 H([u]_\al, [v]_\al) ^p  \; d\alpha  \right)^{1/p}
\end{equation}
for all $u,v\in S^m$.
Note that $u\in F( \mathbb{R}^m)$ satisfies assumption (\romannumeral2) is equivalent
to
$[u]_\al \in C( \mathbb{R}^m)$ for all $\al\in (0,1]$.
So
$d_p$ distance ($1\leq p<\infty$)
can be defined
on
$ F_{GB}(  \mathbb{R}^m ) $.
But $d_p$ distance is not a metric
 on
$F_{GB}   (  \mathbb{R}^m ) $ because $d_p(u,v)$ may equal $+\infty$ for some $u,v \in F_{GB}   (  \mathbb{R}^m ) $.
It's easy to check
that
$d_p$ distance, $p\geq 1$,
is a metric on
$F_{B}   (  \mathbb{R}^m )^p $.
All the fuzzy sets spaces
mentioned
in this paper
are subspaces of
$(F_{B}   (  \mathbb{R}^m )^p,    d_p) $.

A
kind of $L_p$--type noncompact fuzzy sets space endowed with $d_p$ metric
 is called
\emph{$L_p$-extension} of the corresponding compact fuzzy sets space with $d_p$ metric.
So
 fuzzy sets spaces
$(E^{m,p},d_p)$,  $(S_0^{m,p}, d_p )$, $(S^{m,p}, d_p )$, $(\widetilde{S}^{m,p}, d_p)$
and
 $(F_{B}   (  \mathbb{R}^m )^p, d_p) $
are
the
$L_p$-extensions
of
fuzzy sets spaces
$(E^{m},d_p)$, $(S_0^{m}, d_p )$, $(S^{m}, d_p )$, $(\widetilde{S}^{m}, d_p)$
and
$(F_{B}   (  \mathbb{R}^m ), d_p) $,
respectively.

Diamond and Kloeden \cite{da}
pointed out that $(E^m, d_p)$ is not a complete space.
Their analysis
also
indicates
that
the four spaces $(S_0^{m}, d_p )$,
 $(S^m, d_p)$, $( \widetilde{S}^m, d_p   )$ and $(F_{B}   (  \mathbb{R}^m ), d_p) $
are
not
complete.
Kr\"{a}tschmer \cite{kratschmer} has given the completion of $(E^m, d_p)$
which is described by using
 support functions of fuzzy numbers.

In the sequel of
 this paper, we show
that
the
completion of every incomplete fuzzy sets space mentioned in this paper is
exactly its $L_p$-extension , i.e.,
their completions can be obtained
by
means of $L_p$-extension.

\begin{re} \label{act}
{\rm    It can be checked that for each $u\in  F_{B}   (  \mathbb{R}^m )^p   $,
$[u]_\al$ is a compact set when $\alpha\in (0,1]  $,
and
$[u]_0$, the 0-cut,
is
the
only possible unbounded cut-set.
So
we know that $[u]_\al\in K_C( \mathbb{R}^m)$ for all $u\in E^{m,p}$ and $\alpha\in (0,1]$,
and
that
$[u]_\al\in K_S( \mathbb{R}^m)$ for all $u\in \widetilde{S}^{m,p}$ and $\alpha\in (0,1]$.
}
\end{re}

  Denote $\mbox{ker}\; u: =\bigcap_{\alpha \in (0,1]} \mbox{ker}\; [u]_\al$
for
 $u\in \widetilde{S}^{m,p}$ (also see \cite{da, qiu}).
It is easy to check that,
given $u\in \widetilde{S}^{m,p}$,
then
 $u\in S^{m,p}$ if and only if $\mbox{ker}\; u \not= \emptyset$.

The following representation theorem is used
widely in the theory of fuzzy numbers.

\begin{pp}\cite{nr} \label{nr}\
Given $u\in E^m,$ then
\\
(\romannumeral1) \  $[u]_\la\in K_C(\mathbb{R}^m)$ for all $\la\in [0,1]$;
\\
(\romannumeral2) \ $[u]_\la=\bigcap_{\gamma<\lambda}[u]_\gamma$ for all $\la\in (0,1]$;
\\
(\romannumeral3) \ $[u]_0=\overline{\bigcup_{\gamma>0}[u]_\gamma}$.

Moreover, if the family of sets $\{v_\al:\al\in [0,1]\}$ satisfy
conditions $(\romannumeral1)$ through $(\romannumeral3)$ then there exists a unique $u\in E^m$
such that $[u]_{\la}=v_\lambda$ for each $\la\in [0,1].$
\end{pp}

Similarly, we can obtain representation theorems for $S_0^m$, $S^m$, $\widetilde{S}^m$, $F_{B}   (  \mathbb{R}^m ) $,
$E^{m,p}$, $S_0^{m,p}$, $S^{m,p}$, $\widetilde{S}^{m,p}$ and $F_{B}   (  \mathbb{R}^m )^p$ which are listed below and will be used in the sequel.

\begin{tm} \label{rs0lp}
  Given $u\in S_0^{m},$ then
\\
(\romannumeral1) \ $[u]_\la\in K_S(\mathbb{R}^m)$ for all $\la\in [0,1]$,
and
$0 \in \bigcap_{\al\in (0,1]} \mbox{ker}\; [u]_\alpha $;
\\
(\romannumeral2) \ $[u]_\la=\bigcap_{\gamma<\lambda}[u]_\gamma$ for all $\la\in (0,1]$;
\\
(\romannumeral3) \  $[u]_0=\overline{\bigcup_{\gamma>0}[u]_\gamma}$.

Moreover, if the family of sets $\{v_\al:\al\in [0,1]\}$ satisfy
conditions $(\romannumeral1)$ through $(\romannumeral3)$, then there exists a unique $u\in S_0^{m}$
such that $[u]_{\la}=v_\lambda$ for each $\la\in [0,1]$.
\end{tm}

\begin{tm} \label{rslp}
  Given $u\in S^{m},$ then
\\
(\romannumeral1) \ $[u]_\la\in K_S(\mathbb{R}^m)$ for all $\la\in [0,1]$,
and
$\bigcap_{\al\in (0,1]} \mbox{ker}\; [u]_\alpha \not= \emptyset$;
\\
(\romannumeral2) \ $[u]_\la=\bigcap_{\gamma<\lambda}[u]_\gamma$ for all $\la\in (0,1]$;
\\
(\romannumeral3) \  $[u]_0=\overline{\bigcup_{\gamma>0}[u]_\gamma}$.

Moreover, if the family of sets $\{v_\al:\al\in [0,1]\}$ satisfy
conditions $(\romannumeral1)$ through $(\romannumeral3)$, then there exists a unique $u\in S^{m}$
such that $[u]_{\la}=v_\lambda$ for each $\la\in [0,1]$.
\end{tm}

\begin{tm} \label{rgslp}
  Given $u\in \widetilde{S}^{m},$ then
\\
(\romannumeral1) \  $[u]_\la\in K_S(\mathbb{R}^m)$ for all $\la\in [0,1]$;
\\
(\romannumeral2) \ $[u]_\la=\bigcap_{\gamma<\lambda}[u]_\gamma$ for all $\la\in (0,1]$;
\\
(\romannumeral3) \  $[u]_0=\overline{\bigcup_{\gamma>0}[u]_\gamma}$.

Moreover, if the family of sets $\{v_\al:\al\in [0,1]\}$ satisfy
conditions $(\romannumeral1)$ through $(\romannumeral3)$, then there exists a unique $u\in \widetilde{S}^{m}$
such that $[u]_{\la}=v_\lambda$ for each $\la\in [0,1]$.
\end{tm}

\begin{tm} \label{rfbe}
  Given $u\in     F_{B}   (  \mathbb{R}^m ) ,$ then
\\
(\romannumeral1) \ $[u]_\la\in K(\mathbb{R}^m)$ for all $\la\in [0,1]$;
\\
(\romannumeral2) \ $[u]_\la=\bigcap_{\gamma<\lambda}[u]_\gamma$ for all $\la\in (0,1]$;
\\
(\romannumeral3) \  $[u]_0=\overline{\bigcup_{\gamma>0}[u]_\gamma}$.

Moreover, if the family of sets $\{v_\al:\al\in [0,1]\}$ satisfy
conditions $(\romannumeral1)$ through $(\romannumeral3)$, then there exists a unique $u\in  F_{B}   (  \mathbb{R}^m ) $
such that $[u]_{\la}=v_\lambda$ for each $\la\in [0,1]$.
\end{tm}

\begin{tm} \label{rep}
  Given $u\in E^{m,p},$ then
\\
(\romannumeral1) \ $[u]_\la\in K_C(\mathbb{R}^m)$ for all $\la\in (0,1]$;
\\
(\romannumeral2) \ $[u]_\la=\bigcap_{\gamma<\lambda}[u]_\gamma$ for all $\la\in (0,1]$;
\\
(\romannumeral3) \  $[u]_0=\overline{\bigcup_{\gamma>0}[u]_\gamma}$;
\\
(\romannumeral4) \  $\left (\int_0^1   H( [u]_\al, \{0\}   )^p  \; d\alpha   \right)^{1/p}  <   +\infty $.

Moreover, if the family of sets $\{v_\al:\al\in [0,1]\}$ satisfy
conditions $(\romannumeral1)$ through $(\romannumeral4)$, then there exists a unique $u\in E^{m,p}$
such that $[u]_{\la}=v_\lambda$ for each $\la\in [0,1]$.
\end{tm}

\begin{tm} \label{rs0ltp}
  Given $u\in S_0^{m,p},$ then
\\
(\romannumeral1) \ $[u]_\la\in K_S(\mathbb{R}^m)$ for all $\la\in (0,1]$
 and
$0 \in \bigcap_{\al\in (0,1]} \mbox{ker}\; [u]_\alpha $;
\\
(\romannumeral2) \ $[u]_\la=\bigcap_{\gamma<\lambda}[u]_\gamma$ for all $\la\in (0,1]$;
\\
(\romannumeral3) \  $[u]_0=\overline{\bigcup_{\gamma>0}[u]_\gamma}$;
\\
(\romannumeral4) \  $\left (\int_0^1   H( [u]_\al, \{0\}   )^p  \; d\alpha   \right)^{1/p}  <   +\infty $.

Moreover, if the family of sets $\{v_\al:\al\in [0,1]\}$ satisfy
conditions $(\romannumeral1)$ through $(\romannumeral4)$, then there exists a unique $u\in S_0^{m,p}$
such that $[u]_{\la}=v_\lambda$ for each $\la\in [0,1]$.
\end{tm}

\begin{tm} \label{rs0p}
  Given $u\in S^{m,p},$ then
\\
(\romannumeral1) \ $[u]_\la\in K_S(\mathbb{R}^m)$ for all $\la\in (0,1]$
 and
$\bigcap_{\al\in (0,1]} \mbox{ker}\; [u]_\alpha \not= \emptyset$;
\\
(\romannumeral2) \ $[u]_\la=\bigcap_{\gamma<\lambda}[u]_\gamma$ for all $\la\in (0,1]$;
\\
(\romannumeral3) \  $[u]_0=\overline{\bigcup_{\gamma>0}[u]_\gamma}$;
\\
(\romannumeral4) \  $\left (\int_0^1   H( [u]_\al, \{0\}   )^p  \; d\alpha   \right)^{1/p}  <   +\infty $.

Moreover, if the family of sets $\{v_\al:\al\in [0,1]\}$ satisfy
conditions $(\romannumeral1)$ through $(\romannumeral4)$, then there exists a unique $u\in S^{m,p}$
such that $[u]_{\la}=v_\lambda$ for each $\la\in [0,1]$.
\end{tm}

\begin{tm}
  \label{rsp}
Given $u\in \widetilde{S}^{m,p},$ then
\\
(\romannumeral1) \ $[u]_\la\in K_S(\mathbb{R}^m)$ for all $\la\in (0,1]$;
\\
(\romannumeral2) \ $[u]_\la=\bigcap_{\gamma<\lambda}[u]_\gamma$ for all $\la\in (0,1]$;
\\
(\romannumeral3) \  $[u]_0=\overline{\bigcup_{\gamma>0}[u]_\gamma}$;
\\
(\romannumeral4) \  $\left (\int_0^1   H( [u]_\al, \{0\}   )^p  \; d\alpha   \right)^{1/p}  <   +\infty $.

Moreover, if the family of sets $\{v_\al:\al\in [0,1]\}$ satisfy
conditions $(\romannumeral1)$ through $(\romannumeral4)$, then there exists a unique $u\in   \widetilde{S}^{m,p}$
such that $[u]_{\la}=v_\lambda$ for each $\la\in [0,1]$.

\end{tm}

\begin{tm}
  \label{rfbp}
Given $u\in F_{B}   (  \mathbb{R}^m )^p,$ then
\\
(\romannumeral1) \ $[u]_\la\in K(\mathbb{R}^m)$ for all $\la\in (0,1]$;
\\
(\romannumeral2) \ $[u]_\la=\bigcap_{\gamma<\lambda}[u]_\gamma$ for all $\la\in (0,1]$;
\\
(\romannumeral3) \  $[u]_0=\overline{\bigcup_{\gamma>0}[u]_\gamma}$;
\\
(\romannumeral4) \  $\left (\int_0^1   H( [u]_\al, \{0\}   )^p  \; d\alpha   \right)^{1/p}  <   +\infty $.

Moreover, if the family of sets $\{v_\al:\al\in [0,1]\}$ satisfy
conditions $(\romannumeral1)$ through $(\romannumeral4)$, then there exists a unique $u\in   F_{B}   (  \mathbb{R}^m )^p$
such that $[u]_{\la}=v_\lambda$ for each $\la\in [0,1]$.

\end{tm}

\section{Characterizations of relatively compact sets, totally bounded sets and compact sets in $(F_{B}   (  \mathbb{R}^m )^p, d_p)$ \label{crfpat}}

In this section, we present a characterization of relatively compact sets in
 fuzzy sets space
$(F_{B}   (  \mathbb{R}^m )^p, d_p)$.
Based on this, we then give
  characterizations of totally bounded sets and compact sets
in $(F_{B}   (  \mathbb{R}^m )^p, d_p)$.

The topic of characterizations of compactness of fuzzy sets spaces with $d_p$ metric
has been studied
for many years.
 The following two concepts are important to establish and illustrate characterizations of compactness,
which are introduced
by
 Diamond and Kloeden \cite{da3} and Ma \cite{ma}.

\begin{df}
  \cite{da3}
  Let $u\in F_{B}   (  \mathbb{R}^m )^p$. If for given $\varepsilon>0$, there is a $\delta(u,\varepsilon)>0$
  such that for all $0\leq h <\delta$
  $$ \left( \int_h^1  H([u]_\alpha,   [u]_{\alpha-h})^p \;d\alpha   \right)^{1/p}  <\varepsilon ,   $$
where  $1\leq p <+\infty$, then we say $u$ is $p$-mean left-continuous.

Suppose that $U$ is a nonempty set in $F_{B}   (  \mathbb{R}^m )^p$.
If the above inequality holds uniformly for all $u\in U$,
then
we say $U$ is  $p$-mean equi-left-continuous.
\end{df}

\begin{df}
   \cite{ma}
   A set $U\subset F_{B}   (  \mathbb{R}^m )^p$
   is said to be uniformly $p$-mean bounded if there is a constant $M>0$
   such that
   $d_p(u, \widehat{0})\leq M$
   for all $u\in U$.

\end{df}
   It is easy to check that $U$ is uniformly $p$-mean bounded is equivalent to $U$ is a
   bounded set in $( F_{B}   (  \mathbb{R}^m )^p, d_p)$.

Diamond \cite{da2,da3} characterized the compact sets in $(E^m, d_p)$ and $(S^m_0, d_p)$.

\begin{pp}
\cite{da3} \label{dempc}
  A closed set $U$ of $(E^m, d_p)$ is compact if and only if:
\\
(\romannumeral1) $\{[u]_0 : u\in U\}$ is bounded in $(K(\mathbb{R}^m), H)$;
\\
(\romannumeral2) $U$ is $p$-mean equi-left-continuous.

\end{pp}

\begin{pp}
\cite{da2} \label{dsm0pc}
  A closed set $U$ of $(S^m_0, d_p)$ is compact if and only if:
\\
(\romannumeral1) $\{[u]_0 : u\in U\}$ is bounded in $(K(\mathbb{R}^m), H)$;
\\
(\romannumeral2) $U$ is $p$-mean equi-left-continuous.

\end{pp}

Ma \cite{ma} found that   there exists an error in   Proposition \ref{dempc} and modified it to following
Proposition \ref{mgep}.
 For  the convenience of writing,
let $u\in F(\mathbb{R}^m)$,
Ma
used the symbol $u^{(\alpha)}$ to denote the fuzzy set
\[
u^{(\alpha)}(x):=
\left\{
  \begin{array}{ll}
    u(x), & \hbox{if} \  u(x) \geq\alpha,
\\
    0, & \hbox{if} \   u(x)<\alpha.
  \end{array}
\right.
\]

\begin{pp}
\cite{ma} \label{mgep}
  A closed set $U$ of $(E^m, d_p)$ is compact if and only if:
\\
(\romannumeral1) $U$ is uniformly $p$-mean bounded;
\\
(\romannumeral2) $U$ is $p$-mean equi-left-continuous;
\\
(\romannumeral3)
For $\{u_k\} \subset U$,   if $\{u_k^{(h)} : k=1,2,\ldots \}$
converges to
$u(h) \in E^m $ in $d_p$ metric for any $h>0$, then there exists a $u_0 \in E^m$
such that
$u_0^{(h)}=u(h)$.
\end{pp}

Wu and Zhao \cite{wu2} pointed out there still exists fault
 in compactness characterization given by
Ma \cite{ma}. They showed that   Propositions \ref{dempc}, \ref{dsm0pc} and \ref{mgep} are all wrong
and
then
gave
right forms of
compactness criteria for $(E^m, d_p)$ and $(S^m_0, d_p)$.
\begin{pp}
\cite{wu2} \label{gep}
  A closed set $U$ of $(E^m, d_p)$ is compact if and only if:
\\
(\romannumeral1) $U$ is uniformly $p$-mean bounded;
\\
(\romannumeral2) $U$ is $p$-mean equi-left-continuous;
\\
(\romannumeral3) Let $r_i$ be a decreasing sequence in $(0,1]$
converging to zero.
For $\{u_k\} \subset U$,   if $\{u_k^{(r_i)} : k=1,2,\ldots \}$
converges to
$u(r_i) \in E^m $ in $d_p$ metric, then there exists a $u_0 \in E^m$
such that
$[u_0^{(r_i)}]_\al=[u(r_i)]_\al$ when $r_i < \alpha \leq 1$.
\end{pp}

\begin{pp} \cite{wu2} \label{gs0p}
  A closed set $U$ of $(S^m_0, d_p)$ is compact if and only if:
\\
(\romannumeral1) $U$ is uniformly $p$-mean bounded;
\\
(\romannumeral2) $U$ is $p$-mean equi-left-continuous;
\\
(\romannumeral3) Let $r_i$ be a decreasing sequence in $(0,1]$
converging to zero.
For $\{u_k\} \subset U$,   if $\{u_k^{(r_i)} : k=1,2,\ldots \}$
converges to
$u(r_i) \in S^m_0 $ in $d_p$ metric, then there exists a $u_0 \in S^m_0$
such that
$[u_0^{(r_i)}]_\al=[u(r_i)]_\al$ when $r_i < \alpha \leq 1$.
\end{pp}

Based on Proposition \ref{gs0p},
 Zhao and Wu \cite{zhao} further presented
 compactness criteria of $(S^m, d_p)$.

\begin{pp}\cite{zhao}\label{gsp}
  A closed set $U$ of $(S^m, d_p)$ is compact if and only if:
\\
(\romannumeral1) $U$ is uniformly $p$-mean bounded;
\\
(\romannumeral2) $U$ is $p$-mean equi-left-continuous;
\\
(\romannumeral3) Let $r_i$ be a decreasing sequence in $(0,1]$
converging to zero.
For $\{u_k\} \subset U$,   if $\{u_k^{(r_i)} : k=1,2,\ldots \}$
converges to
$u(r_i) \in S^m $ in $d_p$ metric, then there exists a $u_0 \in S^m$
such that
$[u_0^{(r_i)}]_\al=[u(r_i)]_\al$ when $r_i < \alpha \leq 1$.
\end{pp}

Compare Propositions \ref{gep}, \ref{gs0p}, \ref{gsp}
with
Arzel\`{a}--Ascoli theorem,
we notice that
the latter
provides the compactness criteria by characterizing the totally bounded set
 while the former does not seem to do so.
Since totally boundedness   is a key feature of compactness,
it
is
a basic and important problem
to
consider characterizations of totally bounded sets in fuzzy sets spaces.
In order to obtain characterization of totally bounded sets in
$(F_{B}   (  \mathbb{R}^m )^p, d_p)$,
we
first give
a
characterization of
relatively compact sets in $(F_{B}   (  \mathbb{R}^m )^p, d_p)$.

Some fundamental conclusions and concepts in classic analysis and topology
are listed below, which are useful in this paper.
 The readers can see \cite{bridges} for details.

\begin{itemize}
\item \textbf{Lebesgue's Dominated Convergence Theorem}.
 Let $\{f_n\}$ be a sequence of integrable functions
that
 converges almost everywhere to a function $f$,
 and
suppose that  $\{f_n\}$ is dominated by an integrable function $g$.
Then $f$ is integrable,
 and
$ \int f = \lim_{n\to\infty} \int f_n$.

\item \textbf{Fatou's Lemma}.
 Let $\{f_n\}$ be a sequence of nonnegative integrable functions
that
 converges almost everywhere to a function $f$,
and
if the sequence $\{\int f_n\}$ is bounded above,
then
$f$ is integrable and
$\int f \leq   \liminf \int f_n$

\item
\textbf{Absolute continuity of Lebesgue integral}.
 Suppose
that
$f$ is Lebesgue integrable on $E$, then
for arbitrary $\varepsilon>0$,
there is a $\delta>0$
such that
$\int_A  f \;dx < \varepsilon$
 whenever $A \subseteq  E$ and $m(A)<\delta$.

\item
\textbf{Minkowski's inequality}.
Let $p \geq 1$, and let $f,g$ be measurable functions on $\mathbb{R}$
such that
$|f|^p$ and $|g|^p$ are integrable.
Then
$|f+g|^p$ is integrable, and Minkowski's inequality
$$ \left( \int |f+g|^p   \right)^{1/p}   \leq   \left( \int |f|^p   \right)^{1/p}  +     \left( \int |g|^p   \right)^{1/p} $$
holds.

  \item  A relatively compact     subset $Y$ of a topological space X is a subset whose closure is compact.
 In the case of a metric topology, 
the criterion for relative compactness becomes that any sequence in $Y$ has a subsequence convergent in $X$.

  \item Let $(X,d)$ be a metric space.
A
 set $U$ in $X$ is totally bounded
if and only if
for each $\varepsilon>0$,
it contains a finite $\varepsilon$-approximation,
where
an $\varepsilon$-approximation to $U$ is a subset $S$ of $U$
such that
$\rho(x,S)<\varepsilon$
for each $x\in U$.

\item Let $(X,d)$ be a metric space.
Then
a set $U$ in $X$ is relatively compact implies that it is  totally bounded.
For subsets of a complete metric space these two meanings coincide.
Thus
 $(X,d)$ is a compact space iff $X$ is  totally bounded and complete.
 \end{itemize}

We start with some lemmas which discuss some properties of bounded subsets and elements of $(F_{B}   (  \mathbb{R}^m )^p, d_p)$.

\begin{lm}\label{Unes}
  Suppose that $U$ is a bounded set in $(F_{B}   (  \mathbb{R}^m )^p, d_p)$,
  then
  $\{[u]_\al:  u\in U\}$ is a bounded set in $(K(\mathbb{R}^m), H)$ for each $\alpha>0$.
\end{lm}

\begin{proof}  \ If
 there exists an $\al_0>0$ such that
$\{[u]_{\al_0}:  u\in U\}$ is   not a bounded set in $(K(\mathbb{R}^m), H)$.
Then
there is a $u\in U$ such that
$[u]_{\al_0} \notin K(\mathbb{R}^m)$
or
$\{[u]_{\al_0}:  u\in U\}$ is a unbounded set in $ K(\mathbb{R}^m)$.
For both cases,
there
exist $u_n\in U$
such
that
$H([u_n]_{\al_0},    \{0\})  >   n\cdot (\frac{1}{\al_0})^{1/p}$
when $n=1,2,\ldots$,
and hence
$ \left( \int_0^1   H( [u_n]_{\al},    \{0\}  )^p  \; d\alpha  \right)^{1/p}  > n,$
which
contradicts
the
boundness
of
$U$.
\end{proof}

\begin{lm} \label{lcf}
  Suppose that $u\in F_{B}   (  \mathbb{R}^m )^p$ and $\al\in (0,1]$,
  then
  $H([u]_\al, [u]_{\beta}) \to 0$ as $\beta\to \alpha-$.
\end{lm}

\begin{proof} \ The desired result
follows
immediately from
Theorem \ref{rfbp} and Proposition \ref{mce}.
 \end{proof}

\begin{lm} \label{uplc}
If  $u\in F_{B}   (  \mathbb{R}^m )^p$,
 then
  $u$ is $p-$mean left-continuous.
\end{lm}

\begin{proof}
 \ Given $\varepsilon>0$. Note that $u\in F_{B}   (  \mathbb{R}^m )^p$,
from the absolute continuity of Lebesgue integral,
we know
there exists an
$h_1>0$ such that
\begin{equation}\label{h1e}
  \left(\int_0^{h_1}    H([u]_\alpha, \{0\})^p \;d\alpha  \right)^{1/p}   \leq   \varepsilon/3.
\end{equation}
Note
that
  $[u]_{\frac{h_1}{2}} \in K(\mathbb{R}^m)$,
then
there exists an $M>0$ such that $H([u]_{\frac{h_1}{2}}, \{0\})<M$.
This
yields
 that
 \begin{equation}\label{ube}
    H([u]_\alpha,  [u]_{\alpha-h} )   \leq   H([u]_\alpha, \{0\})  +  H( [u]_{\alpha-h}, \{0\} ) \leq 2M
 \end{equation}
for all     $\alpha\geq h_1$ and $h \leq h_1/2$.

By Lemma \ref{lcf},we know that $H([u]_\alpha,    [u]_{\alpha-h}) \rightarrow 0$ when $h\to 0+$,
it
then
follows from the Lebesgue's dominated convergence theorem and \eqref{ube}
that
\begin{equation*}
\left(  \int_{h_1}^1     H([u]_\alpha,  [u]_{\alpha-h} )^p \;d\alpha \right)^{1/p}  \rightarrow   0
\end{equation*}
when $h \to 0+$. Thus there exists an $h_2>0$ such that
\begin{equation}\label{hug}
 \left(  \int_{h_1}^1     H([u]_\alpha,  [u]_{\alpha-h} )^p    \;d\alpha  \right)^{1/p}  <   \varepsilon/3
\end{equation}
for all $0 \leq h\leq \min\{h_1/2, h_2\}$.

Now combined \eqref{h1e} and \eqref{hug}, we know that, for all $h\leq h_3= \min\{h_1/2, h_2\}$,
\begin{align}
     & \left(\int_h^1   H([u]_\alpha,  [u]_{\alpha-h} )^p \;d\alpha   \right)^{1/p}    \nonumber
\\
   & \leq   \left(\int_h^{h_1}   H([u]_\alpha,  [u]_{\alpha-h} )^p \;d\alpha   \right)^{1/p}
     +
    \left(\int_{h_1}^{1}   H([u]_\alpha,  [u]_{\alpha-h} )^p \;d\alpha   \right)^{1/p}   \nonumber
     \\
  &\leq  \left(\int_h^{h_1}   H([u]_\alpha, \{0\} )^p \;d\alpha   \right)^{1/p}
      +
      \left(\int_h^{h_1}     H( [u]_{\alpha-h} ,   \{0\} )^p     \;d\alpha   \right)^{1/p}
      +
      \varepsilon/3                    \nonumber
  \\
   &\leq \left(\int_0^{h_1}   H([u]_\alpha, \{0\} )^p \;d\alpha   \right)^{1/p}
      +
      \left(\int_0^{h_1}     H( [u]_{\alpha} ,   \{0\} )^p     \;d\alpha   \right)^{1/p}
      +
      \varepsilon/3  \nonumber
    \\
     &\leq  \varepsilon/3  +   \varepsilon/3 +   \varepsilon/3 =\varepsilon.
\end{align}
From the arbitrariness of $\varepsilon$, we know that $u$ is $p-$mean left-continuous.
\end{proof}

Now, we arrive at the main results of this section.
The
following
theorem gives a characterization of relatively compactness in   $(F_{B}   (  \mathbb{R}^m )^p,   d_p)$.

\begin{tm}\label{pcn}
 $U$ is a relatively compact set in $(F_{B}   (  \mathbb{R}^m )^p,   d_p)$
 if and only if
 \\
 (\romannumeral1) \  $U$ is uniformly $p$-mean bounded;
 \\
(\romannumeral2) \ $U$ is $p$-mean equi-left-continuous.
\end{tm}

\begin{proof}
\ \emph{\textbf{Necessity.}}
If $U$ is a relatively compact set in
$(F_{B}   (  \mathbb{R}^m )^p,   d_p)$.
 Since $(F_{B}   (  \mathbb{R}^m )^p,   d_p)$ is a metric space,
it
 follows immediately that $U$ is a
bounded set in $(F_{B}   (  \mathbb{R}^m )^p,   d_p)$,
i.e.
$U$ is uniformly $p$-mean bounded.

Now we prove that $U$ is   $p-$mean equi-left-continuous.
Given $\varepsilon>0$. Since $U$ is a relatively compact set,
there exists an $\varepsilon/3$-net $\{u_1,u_2,\ldots u_n\}$ of
$U$.
From Lemma \ref{uplc}, we know that $\{ u_k: k=1,2,\ldots n \}$
is $p$-mean equi-left-continuous.
Hence
there exists $\delta>0$ such that
\begin{equation}\label{uk1e}
  \left(\int_h^{1}    H([u_k]_\alpha, [u_k]_{\alpha-h} )^p \;d\alpha  \right)^{1/p}   \leq   \varepsilon/3.
\end{equation}
for all $h\in [0,\delta)$
and
$k=1,2,\ldots, n$.

Given $u\in U$, there is an $u_k$ such that $d_p(u, u_k)\leq \varepsilon/3$,
and
thus,
by \eqref{uk1e}, we know that for all $h\in [0,\delta)$,
\begin{align}
      &  \left(\int_h^1   H([u]_\alpha,  [u]_{\alpha-h} )^p \;d\alpha   \right)^{1/p}    \nonumber
\\
   &  \leq      \left(\int_h^1   H([u]_\alpha,  [u_k]_{\alpha} )^p \;d\alpha   \right)^{1/p}
   +
    \left(\int_h^1   H([u_k]_\alpha,  [u_k]_{\alpha-h} )^p \;d\alpha   \right)^{1/p}
   +
    \left(\int_h^1   H([u_k]_{\alpha-h},  [u]_{\alpha-h} )^p \;d\alpha   \right)^{1/p}
    \nonumber
     \\
  &    \leq  \varepsilon/3
  +
   \varepsilon/3
      +
      \varepsilon/3
 =\varepsilon.
\end{align}
From
the
 arbitrariness of $\varepsilon$ and $u\in U$, we obtain that $U$ is $p$-mean equi-left-continuous.

\emph{\textbf{Sufficiency.}} If $U$ satisfies (\romannumeral1) and (\romannumeral2).
To show $U$ is a relatively compact set, it suffices to find a convergent subsequence of an arbitrarily
given sequence in $U$.

Let $\{u_n\}$ be a sequence in $U$.
To find a subsequence $\{v_n\}$ of $\{u_n\}$
which converges to $v \in F_{B}   (  \mathbb{R}^m )^p$ according to $d_p$ metric,
we split the proof into three steps.

\textbf{Step 1.} \ Find a subsequence $\{v_n\}$ of $\{u_n\}$
and $v \in F_{GB}   (  \mathbb{R}^m )$
such that
\begin{equation}\label{vnec}
  H([v_n]_\alpha, [v]_\alpha) \st{   \mbox{a.e.}   }{\to}   0 \ (  [0,1]    ).
\end{equation}

Since $U$   is uniformly $p-$mean bounded,
by
 Lemma \ref{Unes},
  $\{[u]_\alpha: u\in U\}$
   is a bounded set in $K(\mathbb{R}^m)$
   for each $\al\in (0,1]$.
Thus, by Proposition \ref{sca}, for each $\alpha>0$,
$\{[u]_\alpha: u\in U\}$ is a relatively compact set in
$(K(\mathbb{R}^m), H)$.

Arrange
all rational numbers in
$(0,1]$
into a sequence
 $q_1,q_2,\ldots, q_n,\ldots$.
Then
 $\{u_n\}$
has a subsequence $\{u_n^{(1)}\}$ such that
$\{[u_n^{(1)}]_{q_1}\}$
converges to $u_{q_1} \in K(\mathbb{R}^m)$,
i.e. $H(   [u_n^{(1)}]_{q_1},     u_{q_1}   )  \to 0$.
If   $\{u_n^{(1)}\}, \ldots, \{u_n^{(k)}\}$ have been chosen,
we
can
choose
a subsequence $\{    u_n^{  (k+1)  }    \}$
of
$\{   u_n^ { (k) }   \}$
such that
 $\{[u_n^{(k+1)}]_{   q_{k+1}   }\}$
converges to
$u_{   q_{k+1}    }
 \in
 K (\mathbb{R}^m)
 $.
Thus we obtain nonempty compact sets $u_{q_k}, k=1,2,\ldots$.
with
 $ u_{q_m} \subseteq   u_{q_l}$  whenever $q_m > q_l $.

Put $v_n=\{    u_n^{  (n) }    \}$ for $n=1,2,\ldots$.
Then $\{v_n\}$
is
a subsequence of
$\{u_n\}$
and
\begin{equation}\label{qpc}
  H([v_n]_{q_k},   u_{q_k})  \to 0 \ \ \ \mbox{as} \ \ \ n\to \infty
\end{equation}
for   $k=1,2,\ldots$.
Define $\{v_\al:  \al\in [0,1] \} $ as follows:
\[
v_\alpha
=
\left\{
  \begin{array}{ll}
    \bigcap_{q_k  <   \alpha}  u_{q_k}, & \   \alpha \in (0,1]\hbox{;} \\
   \overline{ \bigcup_{\alpha \in (0,1]} v_\alpha}  , &  \  \alpha=0 \hbox{.}
  \end{array}
\right.
\]
Then $v_\al$, $\al\in [0,1]$, have the following properties:
\begin{enumerate}\renewcommand{\labelenumi}{(\roman{enumi})}

  \item  $v_\la\in K(\mathbb{R}^m)$ for all $\la\in (0,1]$;

  \item  $v_\la=\bigcap_{\gamma<\lambda}v_\gamma$ for all $\la\in (0,1]$;

  \item  $v_0=\overline{  \bigcup_{\gamma>0}v_\gamma   }$.

\end{enumerate}
In fact, by Proposition \ref{mce},
we obtain that
$v_\al \in K ( \mathbb{R}^m )$
for all $\alpha \in (0,1]$.
Thus property (\romannumeral1) is proved.
Properties (\romannumeral2) and (\romannumeral3) follow immediately from
the
definition of $v_\al$.

Define   a function
$v: \mathbb{R}^m \to [0,1]$
by
$$v(x)=\left\{
         \begin{array}{ll}
           \bigvee_{x\in v_\lambda}\lambda, & \ x\in \bigcup_{\lambda>0}    v_\lambda,\\
           0, & \ \hbox{otherwise.}
         \end{array}
       \right.
$$
Then $v$ is a fuzzy set on $\mathbb{R}^m$.
From properties (\romannumeral1), (\romannumeral2) and (\romannumeral3) of $v_\al$,
we know that
$$[v]_\al=v_\al.$$
So
 $v\in F_{GB} (\mathbb{R}^m)$.
Clearly if the following statements (\uppercase\expandafter{\romannumeral1}) and (\uppercase\expandafter{\romannumeral2})
are
true, then
we obtain
\eqref{vnec}, i.e.  $ H([v_n]_\alpha, [v]_\alpha) \st{   \mbox{a.e.}   }{\to}   0 \ (  [0,1]    )$.
\begin{enumerate}\renewcommand{\labelenumi}{(\Roman{enumi})}
  \item $P(v)$ is at most countable, where $P(v)=\{ \al\in (0,1):  \overline{\{ v>\alpha\}} \subsetneq   [v]_\al  \}$, where
$
\overline{\{ v>\alpha\}} : =\overline{\bigcup_{\beta>\alpha} [v]_{\beta}    }   $.

  \item If $\alpha\in (0,1)\backslash P(v)$,
then
\begin{equation}\label{aehc}
  H([v_n]_\alpha, [v]_\alpha) \to  0  \  \hbox{as} \  n\to \infty.
\end{equation}

\end{enumerate}

Firstly, we show assertion (\uppercase\expandafter{\romannumeral1}).
Let $D(v):=\{\alpha\in (0,1):    [v]_\alpha \nsubseteq  \overline{ \{v> \alpha\}}  \; \}$.
Notice that
$P(v) \subseteq D(v)  $ (In fact, it can be checked that $P(v) = D(v)$).
By the conclusion in Appendix of \cite{huang9},
$D(v) $ is at most countable.
So
$P(v)$ is at most countable.

Secondly, we show assertion (\uppercase\expandafter{\romannumeral2}).
Suppose that  $\alpha\in (0,1)\backslash P(v)$, then from Proposition \ref{mce},
$H( [v]_\beta,  [v]_\al  ) \to 0$ as $\beta \to \alpha$.
Thus,
 given $\varepsilon>0$, we can find a $\delta>0$
such that
$H(u_q, v_\al)<\varepsilon$
for all $q\in \mathbb{Q}$ with $|q-\alpha|<\delta$.
So
\begin{equation*}
  H^*([v_n]_\alpha,   v_\al)  \leq    H^*([v_n]_{q_1},   v_\al)  \leq     H^*([v_n]_{q_1},   u_{q_1}) +\varepsilon
\end{equation*}
for $q_1\in \mathbb{Q} \cap (\alpha-\delta,   \alpha)$.
Hence, by \eqref{qpc} and the arbitrariness of $\varepsilon$,
we obtain
\begin{equation}\label{lhc}
   H^*([v_n]_\alpha,   v_\al) \to 0 \ (n\to \infty).
\end{equation}
On the other hand,
\begin{equation*}
  H^*(v_\alpha,   [v_n]_\al)  \leq    H^*(v_\al,   [v_n]_{q_2}   )  \leq     H^*( u_{q_2},   [v_n]_{q_2}  ) +\varepsilon
\end{equation*}
for $q_2\in \mathbb{Q} \cap (\alpha,  \alpha + \delta)$.
Hence, by \eqref{qpc} and the arbitrariness of $\varepsilon$,
we obtain
\begin{equation}\label{rhc}
   H^*(  v_\al ,    [v_n]_\alpha ) \to 0 \ (n\to \infty).
\end{equation}
Combined with \eqref{lhc} and \eqref{rhc},
we thus
obtain \eqref{aehc}.

\textbf{Step 2. }
Prove
that
\begin{equation}\label{vnc}
   \left(      \int_0^1  H( [v_n]_\alpha,   [v]_\alpha    )^p   \;d\alpha       \right)^{1/p}      \to 0.
\end{equation}

Given
$\varepsilon>0$.
It can be deduced that, for all $h<1/2$,
\begin{align}
   &  \left(      \int_0^h  H( [v_n]_\alpha,  [v]_\alpha   )^p   \;d\alpha       \right)^{1/p}   \nonumber
\\
 \leq & \left(      \int_0^h  H( [v_n]_\alpha,  [v_n]_{\alpha+h}   )^p   \;d\alpha       \right)^{1/p}
 +
  \left(      \int_0^h  H( [v_n]_{\alpha+h},  [v]_{\alpha+h}   )^p   \;d\alpha       \right)^{1/p}
    +
   \left(      \int_0^h  H( [v]_{\alpha+h},  [v]_\alpha   )^p   \;d\alpha       \right)^{1/p}  \nonumber
\\
 = &
   \left(      \int_h^{2h}  H( [v_n]_{\beta-h},  [v_n]_{\beta}   )^p   \;d\beta       \right)^{1/p}
    +
  \left(      \int_h^{2h}  H( [v_n]_{\beta},  [v]_{\beta}   )^p   \;d\beta       \right)^{1/p}
     +
   \left(      \int_h^{2h}  H( [v]_{\beta},  [v]_{\beta-h}   )^p   \;d\beta       \right)^{1/p} \nonumber
   \\
 \leq   &
       \left(      \int_h^{1}  H( [v_n]_{\beta-h},  [v_n]_{\beta}   )^p   \;d\beta       \right)^{1/p}
    +
  \left(      \int_h^{1}  H( [v_n]_{\beta},  [v]_{\beta}   )^p   \;d\beta       \right)^{1/p}
     +
   \left(      \int_h^{1}  H( [v]_{\beta},  [v]_{\beta-h}   )^p   \;d\beta       \right)^{1/p}.\label{ean}
\end{align}
Since $U$ is $p$-mean equi-left-continuous, there exists an $h\in (0,1/2)$
such that
\begin{equation}\label{evhn}
   \left(      \int_h^{1}  H( [v_n]_{\beta-h},  [v_n]_{\beta}   )^p   \;d\beta       \right)^{1/p} <\varepsilon/4
\end{equation}
for all $n=1,2,\ldots$.
From \eqref{vnec}, we know
if $n\to \infty$
then
$ H( [v_n]_{\beta-h},  [v_n]_{\beta} )    \to    H( [v]_{\beta-h},  [v]_{\beta} ) $
a.e. on $\beta\in [h,1]$. So, by Fatou Lemma, we have
\begin{equation}\label{hev}
    \left(      \int_h^{1}  H( [v]_{\beta-h},  [v]_{\beta}   )^p   \;d\beta       \right)^{1/p}
    \leq
      \liminf_n     \left(      \int_h^{1}  H( [v_n]_{\beta-h},  [v_n]_{\beta}   )^p   \;d\beta       \right)^{1/p}
      \leq
      \varepsilon/4,
\end{equation}

Note that $[v_n]_h$ and $[v]_h$   are contained in $\{[u]_\alpha: u\in U\}$ which is compact,
it thus follows from the Lebesgue's
dominated convergence theorem
and
\eqref{vnec}
that
\begin{equation*}
   \left(      \int_h^1  H( [v_n]_\alpha,   [v]_\alpha    )^p   \;d\alpha       \right)^{1/p}    \to 0
\end{equation*}
as $n\to \infty$.
Hence there is an $N(h,\varepsilon)$ such that
\begin{equation}\label{bvnrgdpe}
  \left(      \int_h^1  H( [v_n]_\alpha,   [v]_\alpha    )^p   \;d\alpha       \right)^{1/p} \leq \varepsilon/4
\end{equation}
for all $n\geq N$.

Combined with \eqref{ean}, \eqref{evhn}, \eqref{hev}, and \eqref{bvnrgdpe},
it
yields
that
\begin{align*}
 & \left(      \int_0^1  H( [v_n]_\alpha,   [v]_\alpha    )^p   \;d\alpha       \right)^{1/p}
 \\
 & \leq
    \left(      \int_0^h  H( [v_n]_\alpha,   [v]_\alpha    )^p   \;d\alpha       \right)^{1/p}
    +
     \left(      \int_h^1  H( [v_n]_\alpha,   [v]_\alpha    )^p   \;d\alpha       \right)^{1/p}
     \\
  &    \leq
          \varepsilon/4 +   \left(      \int_h^1  H( [v_n]_\alpha,   [v]_\alpha    )^p   \;d\alpha       \right)^{1/p}
          +\varepsilon/4
          + \left(      \int_h^1  H( [v_n]_\alpha,   [v]_\alpha    )^p   \;d\alpha       \right)^{1/p}
          \\
  &
  \leq \varepsilon
\end{align*}
for all $n\geq N$.
Thus we obtain \eqref{vnc} from the arbitrariness of $\varepsilon$.

\textbf{Step 3.} Show that $v\in F_{B}   (  \mathbb{R}^m )^p$.

By \eqref{vnc}, we know
that
there is an $N$
such that
$$ \left(      \int_0^1  H( [v]_\alpha,  [v_N]_\alpha   )^p   \;d\alpha       \right)^{1/p}  <    1,$$
and
then
\begin{align*}
&     \left(      \int_0^1  H( [v]_\alpha,  \{0\}   )^p   \;d\alpha       \right)^{1/p}
 \\
& \leq
 \left(      \int_0^1  H( [v]_\alpha,  [v_N]_\alpha   )^p   \;d\alpha       \right)^{1/p}
+
  \left(      \int_0^1  H( [v_N]_\alpha,  \{0\}   )^p   \;d\alpha       \right)^{1/p}
  \\
  & \leq
 1 +     \left(      \int_0^1  H( [v_N]_\alpha,  \{0\}   )^p   \;d\alpha       \right)^{1/p}
 < +\infty.
\end{align*}
By properties (\romannumeral1),(\romannumeral2) and (\romannumeral3) of $v_\al$ and Theorem \ref{rfbp},
this yields that $v\in F_{B}   (  \mathbb{R}^m )^p$.

From steps 1, 2 \ and 3, we know
that
for arbitrary sequence $\{u_n\}$ of $U$,
there exists
a subsequence $\{v_n\}$ of $\{u_n\}$
which converges to $v\in F_{B}   (  \mathbb{R}^m )^p$.
This means that $U$ is a relatively compact set in $(F_{B}   (  \mathbb{R}^m )^p, d_p)$.
  \end{proof}

By using the characterization of relatively compact sets in $(F_{B}   (  \mathbb{R}^m )^p,   d_p)$
given in Theorem \ref{pcn},
we can derive
characterizations of totally bounded sets and compact sets below
as consequences.

\begin{tm} \label{tcn}
  $U$ is a totally bounded set in $(F_{B}   (  \mathbb{R}^m )^p,   d_p)$
 if and only if
 \\
 (\romannumeral1) \  $U$ is uniformly $p$-mean bounded;
 \\
(\romannumeral2) \ $U$ is $p$-mean equi-left-continuous.
\end{tm}

\begin{proof} \ Notice that   it only use the totally boundedness of $U$
to show
the necessity part of the proof of Theorem \ref{pcn}.
So
the desired conclusion follows immediately from
Theorem \ref{pcn}.
\end{proof}

\begin{tm} \label{gscn} Let
  $U$ be a subset of $ (F_{B}   (  \mathbb{R}^m )^p,   d_p)$,
then $U$ is    compact  in $ (F_{B}   (  \mathbb{R}^m )^p,   d_p)$
 if and only if
 \\
 (\romannumeral1) \  $U$ is uniformly $p$-mean bounded;
 \\
(\romannumeral2) \ $U$ is $p$-mean equi-left-continuous;
\\
 (\romannumeral3) \  $U$ is a closed set in $ (F_{B}   (  \mathbb{R}^m )^p,   d_p)$.
\end{tm}

\begin{proof} \ The desired result follows immediately from Theorem \ref{pcn}.    \end{proof}

\section{Relationship between $(F_{B}   (  \mathbb{R}^m ), d_p)$ and $(F_{B}   (  \mathbb{R}^m )^p, d_p)$
and
properties    of $(F_{B}   (  \mathbb{R}^m ), d_p)$ \label{relation}}

In this section, we show that
$(F_{B}   (  \mathbb{R}^m )^p, d_p)$
is
the completion of
$(F_{B}   (  \mathbb{R}^m ), d_p)$,
and
then
present
characterizations of totally bounded sets, relatively compact sets and compact sets in $(F_{B}   (  \mathbb{R}^m ), d_p)$.

Diamond and Kloeden \cite{da} pointed out that
$(E^m, d_p)$ is not a complete space.
Ma \cite{ma} gave the following   example to show this fact.
Let
\[
u_n(x)=\left\{
         \begin{array}{ll}
           e^{-x}, & \hbox{if} \ 0\leq x \leq n, \\
           0, & \hbox{otherwise,}
         \end{array}
       \right.
 n=1,2,\ldots.
\]
Then
$\{u_n\} $ is a Cauchy sequence in $ (E^1, d_1) \cap  (S_0^1, d_1)  $. Put
\[
u(x)=\left\{
         \begin{array}{ll}
           e^{-x}, & \hbox{if} \ 0\leq x < +\infty, \\
           0, & \hbox{otherwise,}
         \end{array}
       \right.
\]
then it can be checked that $u\in F_{B}   (  \mathbb{R}^1 )^{1}\backslash   F_{B}   (  \mathbb{R}^1 ) $ and $u_n$ converges to $u$ in $d_1$ metric.
Note
 that
$E^1, S_0^1   \subset    S^1    \subset    \widetilde{S}^{1}   \subset  F_{B}   (  \mathbb{R}^1 )$,
it yields that
none of
$(E^1, d_1)$, $(S_0^1, d_1)$, $(S^1, d_1)$, $(\widetilde{S}^1, d_1)$ and $(F_{B}   (  \mathbb{R}^1 ), d_1)$
is
complete.
Along this way, we can show that none of
$(E^m, d_p)$, $(S_0^m, d_p)$, $(S^m, d_p)$, $(\widetilde{S}^m, d_p)$ and $(F_{B}   (  \mathbb{R}^m ), d_p)$
is complete.

First,
we
do
 step by step
to show that the completion of
 $(F_{B}   (  \mathbb{R}^m ), d_p)$
is
exactly its $L_p$-extension
 $(F_{B}   (  \mathbb{R}^m )^p, d_p)$.
Two facts
are
first proved which are exhibited in Theorems \ref{scp} and \ref{sln}.

\begin{tm} \label{scp}
 $(F_{B}   (  \mathbb{R}^m )^p, d_p)$ is a complete space.
\end{tm}

\begin{proof}   \ It   suffices to prove that each Cauchy sequence has a limit in  $(F_{B}   (  \mathbb{R}^m )^p, d_p)$.
Let
$\{u_n:  n\in \mathbb{N}\}$ be a Cauchy sequence in      $(F_{B}   (  \mathbb{R}^m )^p, d_p)$,
we assert
that
$\{u_n:  n\in \mathbb{N}\}$ is a relatively compact set in  $(F_{B}   (  \mathbb{R}^m )^p, d_p)$.

To show this assertion, by Theorem \ref{pcn},
it is equivalent to prove
that
$\{u_n:  n\in \mathbb{N}\}$ is a bounded set in $(F_{B}   (  \mathbb{R}^m )^p, d_p)$
and
that
$\{u_n:  n\in \mathbb{N}\}$    is  $p$-mean  equi-left-continuous.
The former follows immediately from
the fact
that
$\{u_n:  n\in \mathbb{N}\}$
is a Cauchy sequence.

To prove the latter, suppose
 $\varepsilon>0$. Since $\{u_n: n\in \mathbb{N}\}$
is a Cauchy sequence,
there exists an $N\in \mathbb{N}$ satisfies
that
$d_p(u_n,u_m) \leq \varepsilon/3$
for all $n,m \geq N$.
By Lemma \ref{uplc},
$\{u_k :\ 1\leq k \leq N\}$
is $p$-mean equi-left-continuous,
hence
we can find
an
$h>0$
such that
\begin{equation}
   \left(      \int_h^{1}  H( [u_k]_{\al-h},  [u_k]_{\al}   )^p   \;d\al     \right)^{1/p}  \leq \varepsilon/3   \label{nevc}
\end{equation}
for all $1\leq k \leq N$.
If $k> N$, then
\begin{align}
  &\left(      \int_h^{1}  H( [u_k]_{\al-h},  [u_k]_{\al}   )^p   \;d\al     \right)^{1/p}    \nonumber
\\
 &\hspace{-6mm}\leq  \left(      \int_h^{1}  H( [u_k]_{\al-h},   [u_N]_{\al-h}   )^p   \;d\al     \right)^{1/p}
 +
 \left(      \int_h^{1}  H( [u_N]_{\al-h},  [u_N]_{\al}   )^p   \;d\al     \right)^{1/p}
+
 \left(      \int_h^{1}  H( [u_N]_{\al},   [u_k]_{\al}   )^p   \;d\al     \right)^{1/p} \nonumber
\\
 &\hspace{-6mm} \leq \varepsilon/3   +   \varepsilon/3 + \varepsilon/3 =\varepsilon. \label{knvc}
\end{align}
From the arbitrariness of $\varepsilon$
and
ineqs. \eqref{nevc} and \eqref{knvc},
we
know
that
  $\{u_n:   n\in \mathbb{N}\}$ is   $p$-mean equi-left-continuous.

Now, from   the relatively compactness
of
   $\{u_n:   n\in \mathbb{N}\}$ in  $(F_{B}   (  \mathbb{R}^m )^p, d_p)$,
there
 exists
 a subsequence $\{u_{n_k}: k=1,2,\ldots\}$
of
$\{u_n:   n\in \mathbb{N}\}$
such that
$\lim_{k\to \infty} u_{n_k}  =  u \in F_{B}   (  \mathbb{R}^m )^p$.
Note
that
$\{u_n:   n\in \mathbb{N}\}$
is a Cauchy sequence,
we
thus
know that
$u_n$, $n=1,2,\ldots$,
also converges to $u$
in
$(F_{B}   (  \mathbb{R}^m )^p,  d_p)$.

The proof is completed.
\end{proof}

\begin{re}
  {\rm
By Theorems \ref{pcn} and \ref{tcn},
a set $U$ in $(F_{B}   (  \mathbb{R}^m )^p, d_p)$ is  totally bounded if and only if it is relatively compact.
This fact
alone
can ensure that $(F_{B}   (  \mathbb{R}^m )^p, d_p)$ is complete.

}
\end{re}

\begin{tm} \label{sln}
  $F_{B}   (  \mathbb{R}^m )$ is a dense set in $(F_{B}   (  \mathbb{R}^m )^p, d_p)$.
\end{tm}

\begin{proof}  \ Given $u\in F_{B}   (  \mathbb{R}^m )^p$. Put $u_n=u^{(1/n)}$, $n=1,2,\ldots$.
Then
\[
[u_n]_\al=\left\{
      \begin{array}{ll}
        [u]_\al, & \hbox{if} \    \al\geq   1/n,
\\
        \mbox{}[u]_{1/n}, & \hbox{if} \      \alpha \leq 1/n,
           \end{array}
    \right.
\]
for all $\alpha \in [0,1]$.
Notice
that
$[u_n]_0 =[u]_{1/n}\in K(\mathbb{R}^m)$.
 It
thus follows from Theorems \ref{rfbe} and \ref{rfbp}
that
$u_n\in F_{B}   (  \mathbb{R}^m )$ for $n=1,2,\ldots$.

Since $u\in F_{B}   (  \mathbb{R}^m )^p$,
we know
 $\left( \int_0^{1}  H(   [u]_{\al}, \{0 \}   )^p   \;d\al \right)^{1/p} < +\infty$,
thus, by the absolute continuity of the Lebesgue's integral,
it holds that,
for each $\varepsilon>0$,
there is a $\delta (\varepsilon)>0$
such that
\begin{equation}\label{ace}
  \left( \int_0^{\delta}  H(   [u]_{\al}, \{0 \}   )^p   \;d\al \right)^{1/p} < \varepsilon.
\end{equation}
Note that
\begin{align*}
d_p(u_n, u)& = \left( \int_0^{1/n}  H( [u]_{1/n},   [u]_{\al}   )^p   \;d\al \right)^{1/p}
\\
& \leq \left( \int_0^{1/n}  H(  [u]_{\al}, \{0\}   )^p   \;d\al \right)^{1/p}
+
\left( \int_0^{1/n}  H( [u]_{1/n},  \{0\}  )^p   \;d\al \right)^{1/p}
\\
& \leq 2 \left( \int_0^{1/n}  H(  [u]_{\al}, \{0\}   )^p   \;d\al \right)^{1/p},
\end{align*}
it then follows from ineq.\eqref{ace}
that
$d_p(u_n, u) \to 0$ as $n\to \infty$.

So, for each $u\in F_{B}   (  \mathbb{R}^m )^p$, we can find a sequence $\{u_n\} \subset F_{B}   (  \mathbb{R}^m )$
such that
$ u_n$ converges to    $ u$.
This means that $F_{B}   (  \mathbb{R}^m )$
is
dense
in $F_{B}   (  \mathbb{R}^m )^p$.
\end{proof}

\begin{re} \label{undg}
{\rm
From the proof of Theorem \ref{sln}, we know the following fact.
\\
  Given $u\in F_{B}   (  \mathbb{R}^m )^p$, then $u^{(1/n)} \in  F_{B}   (  \mathbb{R}^m )$ for each $n\in \mathbb{N}$,
and
$d_p(u^{(1/n)} ,u) \to 0$ as $n\to \infty$.
}
\end{re}

Now,
 we get relationship between
$ (F_{B}   (  \mathbb{R}^m )^p, d_p)$
and
  $ (F_{B}   (  \mathbb{R}^m ), d_p)$.

\begin{tm}
 $ (F_{B}   (  \mathbb{R}^m )^p, d_p)$ is the completion of  $ (F_{B}   (  \mathbb{R}^m ), d_p)$.
\end{tm}

\begin{proof} \ The desired result follows immediately from Theorems \ref{scp} and \ref{sln}.
 \end{proof}

Next, we exhibit some characterizations of totally bounded sets, relatively compact sets
and
compact sets in $(F_{B}   (  \mathbb{R}^m ), d_p)$ which are consequences of the conclusions in Section 4.

\begin{tm} Let
  $U \subset F_{B}   (  \mathbb{R}^m )$,
then $U$ is     totally bounded  in $ (F_{B}   (  \mathbb{R}^m ),   d_p)$
 if and only if
 \\
 (\romannumeral1) \  $U$ is uniformly $p$-mean bounded;
 \\
(\romannumeral2) \ $U$ is $p$-mean equi-left-continuous.
\end{tm}

\begin{proof}  \
Note that $F_{B}   (  \mathbb{R}^m ) \subset F_{B}   (  \mathbb{R}^m )^p$,
we thus
know
$U\subset F_{B}   (  \mathbb{R}^m ) $ is a    totally bounded    set
in
 $(F_{B}   (  \mathbb{R}^m ),   d_p)$
if and only if
$U$ is a  totally bounded    set in $ (F_{B}   (  \mathbb{R}^m )^p, d_p)$.
So
the desired conclusion follows immediately
from
Theorem \ref{tcn}.
\end{proof}

Given $U \in F_{B}   (  \mathbb{R}^m )^p$,
for
writing convenience,
we use
$\overline{U} $ to denote the closure of $U$ in $(F_{B}   (  \mathbb{R}^m )^p, d_p)$.

\begin{tm} \label{fbrchra} Let
  $U \subset F_{B}   (  \mathbb{R}^m )$,
then $U$ is relatively compact in $ (F_{B}   (  \mathbb{R}^m ),   d_p)$
 if and only if
 \\
 (\romannumeral1) \  $U$ is uniformly $p$-mean bounded;
 \\
(\romannumeral2) \ $U$ is $p$-mean equi-left-continuous;
\\
(\romannumeral3) \  $\overline{U} \subset  F_{B}   (  \mathbb{R}^m )   $.
\end{tm}

\begin{proof}   \ The desired conclusion follows immediately
from
Theorem \ref{pcn} and
the obvious fact
 that
  $      F_{B}   (  \mathbb{R}^m )  \subset F_{B}   (  \mathbb{R}^m )^p$.
\end{proof}

\begin{tm} \label{fbchra} Let
  $U \subset F_{B}   (  \mathbb{R}^m )$,
then $U$ is   compact in $ (F_{B}   (  \mathbb{R}^m ),   d_p)$
 if and only if
 \\
 (\romannumeral1) \  $U$ is uniformly $p$-mean bounded;
 \\
(\romannumeral2) \ $U$ is $p$-mean equi-left-continuous;
\\
(\romannumeral3) \  $\overline{U} = U$.
\end{tm}

\begin{proof} \ The desired result follows immediately from Theorem \ref{gscn}. \end{proof}

Condition (\romannumeral3) in Theorem \ref{fbchra}
involves the closure of
$U$ in the completion space $(F_{B}   (  \mathbb{R}^m )^p, d_p)$.
We
 intend to find another characterization of compactness that depends only on $U$ itself,
which
is
the last result of this section.
To establish
this
new characterization of compactness,
we need
the following concept.

Let $B_r:=\{ x\in \mathbb{R}^m :     \|x\| \leq r    \}$,
where $r$ is a positive real number.
$\widehat{B_r}$ denotes the characteristic function
of $B_r$.
Given $u\in F_{B} (\mathbb{R}^m)$,
then
$u \vee \widehat{B_r}  \in F_{B} (\mathbb{R}^m)$.
 Define
$$    |  u | ^r :  = \left( \int_0^1   H(  [u \vee \widehat{B_r}   ]_\al,       [\widehat{B_r}  ]_\al    ) ^p \, d\al   \right) ^ {1/p}.
$$
It
 can be checked that, for $u\in F_{B} (\mathbb{R}^m)$, $ |  u | ^r =0$ if and only if
$[u]_0 \subseteq  B_r  $.
Note
that
$$  H(  [u \vee \widehat{B_r}  ]_\al,     [ v \vee \widehat{B_r}  ]_\al    ) \leq  H( [u]_\al,   [v]_\al      )     ,$$
it thus holds
that
\begin{equation}\label{nep}
  d_p(u,v) \geq  | |u|^r - |v|^r    |.
\end{equation}

\begin{tm} \label{fbsrchra} Let
  $U \subset F_{B}   (  \mathbb{R}^m )$,
then $U$ is relatively   compact in $ (F_{B}   (  \mathbb{R}^m ),   d_p)$
 if and only if $U$ satisfies
conditions (\romannumeral1), (\romannumeral2) in Theorem \ref{pcn}
and
the following condition (\romannumeral3$'$).
\begin{description}
  \item[(\romannumeral3$'$)]
Given $\{u_n: n=1,2,\ldots\} \subset U$, there exists a $r>0$ and a subsequence $\{v_n\}$ of $\{u_n\}$ such
that
$\lim_{n\to \infty}    |v_n|^r =0$.
\end{description}
\end{tm}

\begin{proof} \
Suppose
that
 $U$ is a relatively compact set but does not satisfy condition (\romannumeral3$'$).
Take
 $r=1$,
then
there exists $\varepsilon_1 >0$ and a subsequence $\{u_{n}^{(1)}: n=1,2,\ldots\}$
of
 $\{u_n: n=1,2,\ldots\}$
such
that
$ |u_n^{(1)}|^1 > \varepsilon_1$ for all $n=1,2,\ldots$.
If
  $\{u_n^{(1)}\}, \ldots, \{u_n^{(k)}\}$ and positive numbers $\varepsilon_1, \ldots,  \varepsilon_k$ have been chosen,
we
can
find
a subsequence $\{    u_n^{  (k+1)  }    \}$
of
$\{   u_n^ { (k) }   \}$
and
$\varepsilon_{k+1} >0$
such that
 $ |u_n^{(k+1)}|^{k+1} > \varepsilon_{k+1}$ for all $n=1,2,\ldots$.
Put $v_n=   u_n^{  (n) }   $ for $n=1,2,\ldots$.
Then $\{v_n\}$
is
a subsequence of
$\{u_n\}$
and
\begin{equation}\label{vne}
 \liminf_{n\to \infty}    |v_n|^k    \geq    \varepsilon_k
\end{equation}
for   $k=1,2,\ldots$.
Let $v \in  F_{B}   (  \mathbb{R}^m )^p$ be a accumulation point of $\{v_n\}$.
It then follows
from \eqref{nep} and \eqref{vne}
that
\begin{equation*}
   |v|^k    \geq   \varepsilon_k > 0
\end{equation*}
for all $k=1,2,\ldots$. So we know $v \notin  F_{B}   (  \mathbb{R}^m ) $.
This contradicts
the fact
that
 $U$ is a relatively compact set in $(  F_{B}   (  \mathbb{R}^m ), d_p  )$.

Suppose that $U \subset (  F_{B}   (  \mathbb{R}^m ), d_p  )$ satisfies condition (\romannumeral3$'$).
Given a sequence
$\{u_n\}$ in
$U$ with $\lim_{n\to \infty} u_n = u \in F_{B}   (  \mathbb{R}^m )^p$.
Then, from
\eqref{nep},
there exists a $r>0$ such
that
$\lim_{n\to \infty}    |u_n|^r =|u|^r =0$.
Hence $[u]_0 \subseteq  B_r   $,
i.e.
$u \in F_{B}   (  \mathbb{R}^m )$.
So,
by Theorem \ref{pcn}, we know
that
 if $U$ meets conditions (\romannumeral1), (\romannumeral2) and (\romannumeral3$'$),
then
$U$ is a relatively compact set in $(  F_{B}   (  \mathbb{R}^m ), d_p)$.
\end{proof}

\begin{tm} \label{fbschra} Let
  $U$ be a set in $ F_{B}   (  \mathbb{R}^m )$,
then $U$ is compact in
$ (F_{B}   (  \mathbb{R}^m ),   d_p)$
 if and only if $U$ is closed in
$ (F_{B}   (  \mathbb{R}^m ),   d_p)$ and satisfies
conditions (\romannumeral1), (\romannumeral2)
and
 (\romannumeral3$'$)
 in Theorem \ref{fbsrchra}.
\end{tm}

\begin{proof} \ The desired result follows immediately from Theorem \ref{fbsrchra}. \end{proof}

\section{Subspaces of $ (F_{B}   (  \mathbb{R}^m )^p,   d_p)$}

We have
already
shown
that
$ (F_{B}   (  \mathbb{R}^m )^p,   d_p)$   is
the
completion
of
 $ (F_{B}   (  \mathbb{R}^m ),   d_p)$.
We
might
expect
that
$ (E^{m,p}, d_p)$, $(S_0^{m,p},d_p)$, $(S^{m,p},d_p)$ and   $(\widetilde{S}^{m,p}, d_p) $
are the completions
of
$ (E^{m}, d_p)$, $(S_0^{m}, d_p)$, $(S^{m}, d_p)$ and   $(\widetilde{S}^{m}, d_p) $, respectively.
In this section, we
prove that this is true
and
discuss
 relationship among all these fuzzy sets spaces.
The conclusions are summarized in
a figure.
Then,
by using these conclusions   and the results in Sections \ref{crfpat} and \ref{relation},
we
gives characterizations of totally boundedness, relatively compactness and compactness
of all subspaces
   of
 $ (F_{B}   (  \mathbb{R}^m )^p,   d_p)$ mentioned in this paper.

First, we give a series of conclusions on relationships among subspaces
of
 $ (F_{B}   (  \mathbb{R}^m )^p,   d_p)$,
which
will
be summarized in a figure.

\begin{tm} \label{gsmc}
$\widetilde{S}^{m,p}$ is a closed set in $ (F_{B}   (  \mathbb{R}^m )^p,   d_p)$.

\end{tm}

\begin{proof} \  It only need to show that each accumulation point of $\widetilde{S}^{m,p}$ belongs to
itself.
Given a sequence $\{u_n\}$ in   $\widetilde{S}^{m,p}$ with $\lim u_n =u \in F_{B}   (  \mathbb{R}^m )^p$,
then
 clearly
$H([u_n]_\alpha, [u]_\alpha) \st{   \mbox{a.e.}   }{\to}   0 \ (  [0,1]    )$.
Suppose
that
$\al \in (0,1]$. If $H([u_n]_\alpha, [u]_\alpha) \to   0 $,
then
by Theorem \ref{ksc},
$[u]_\al \in K_S(\mathbb{R}^m)$.
 If $H([u_n]_\alpha, [u]_\alpha) \not\to   0 $,
then
there exists a sequence
$\beta_n  \to \al-$
such that
$ [u]_{\beta_n} \in  K_S(\mathbb{R}^m)$.
Note that
$[u]_\al= \bigcap_{n} [u]_{\beta_n} $,
this implies
that
$ [u]_\alpha \in  K_S(\mathbb{R}^m)$.
So we have
$u \in \widetilde{S}^{m,p}$.
\end{proof}

\begin{tm} \label{slne}
  $\widetilde{S}^m$ is a dense set in $(\widetilde{S}^{m,p}, d_p)$.
\end{tm}

\begin{proof} \ Given $u\in \widetilde{S}^{m,p}$. Put $u_n=u^{(1/n)}$, $n=1,2,\ldots$.
Then $\{u_n \}\subset \widetilde{S}^m$. From Remark \ref{undg},
we
know
that
$d_p(u_n, u) \to 0$ as $n\to \infty$.
So
 $\widetilde{S}^m$
is
dense
in $(\widetilde{S}^{m,p}, d_p)$. \end{proof}

\begin{tm} \label{sgmcde}
   $S^{m,p}$   is a closed subset of $(\widetilde{S}^{m,p}, d_p)$.
\end{tm}

\begin{proof} \ To show that $S^{m,p}$ is a closed set in $( \widetilde{S}^{m,p},   d_p)$,
let
 $\{u_n\}$ be  a sequence in $S^{m,p}$ which converges to $u \in \widetilde{S}^{m,p}$,
we only need to prove
that
$u\in S^{m,p}$.

Since
$d_p(u_n, u)= \left(    \int_{0}^{1}  H([u_n]_\al,  [u]_\al)^p \; d\alpha    \right)^{1/p} \to 0$,
it holds
that
\begin{equation} \label{hue}
  H([u_n]_\al,  [u]_\al) \to 0 \ \hbox{a.e. on } \ [0,1].
\end{equation}
Hence
 $\{\mbox{ker}\; u_n: n=1,2,\ldots \}$ is a bounded set in $(K(\mathbb{R}^m), H)$,
and therefore
\begin{equation}\label{neu}
 \limsup_{n\to \infty}\mbox{ker}\;  u_n \not= \emptyset.
 \end{equation}
We assert that
 \begin{equation}\label{aceu}
\limsup_{n\to \infty}\mbox{ker}\;  u_n  \subset  \mbox{ker}\;  [u]_\al \ \hbox{for all} \
 \al\in (0,1].
\end{equation}
So, from \eqref{neu} and \eqref{aceu},
 we know
$$
\emptyset \not=\limsup_{n\to \infty}\mbox{ker}\;  u_n  \subset  \bigcap_{\al\in (0,1]} \mbox{ker}\;  [u]_\al = \mbox{ker}\;  u.
$$
It thus follows from
Theorem \ref{rs0p} that
$u\in S^{m,p}$.

Now we prove \eqref{aceu}.
 The
 proof
 is divided into two cases.
\\
\textbf{Case 1}. \ $\alpha\in (0,1]$ satisfies the condition that $H([u_n]_\al,  [u]_\al) \to 0$.
\\
In this case, by
 Corollary \ref{kef}, we have
that
$$
\limsup_{n\to \infty}\mbox{ker}\;  u_n \subset \limsup_{n\to \infty}\mbox{ker}\;  [u_n]_\al \subset \mbox{ker}\; [u]_\al .
$$
\textbf{Case 2}. \  $\alpha\in (0,1]$ satisfies the condition that $H([u_n]_\al,  [u]_\al) \not\to 0$.
\\
By \eqref{hue}, we know that there is a sequence $\alpha_n\in (0,1]$, $n=1,2,\ldots$,
such that
$\alpha_n \to \al-$
and
$H([u_n]_{\al_n},  [u]_{\al_n}) \to 0$.
From case 1, we obtain
that
\begin{equation}\label{kcn}
  \limsup_{n\to \infty}\mbox{ker}\;  u_n \subset \bigcap_{n=1}^{+\infty}   \mbox{ker}\; [u]_{\al_n}.
\end{equation}
Note that
$H( [u]_\al, [u]_{\al_n}     )   \to 0$,
so, by Corollary   \ref{kef},
\begin{equation}\label{anc}
  \limsup_{n\to \infty}\mbox{ker}\;  [u]_{\al_n} \subset   \mbox{ker}\; [u]_{\al},
\end{equation}
combined \eqref{kcn} and \eqref{anc}, we get that
$$ \limsup_{n\to \infty}\mbox{ker}\;  u_n  \subset   \mbox{ker}\; [u]_{\al}.  $$
\end{proof}

\begin{tl}\label{ukc}
  Suppose that $\{u_n\}$ is a sequence in $(S^{m,p}, d_p)$ and that $u\in \widetilde{S}^{m,p}$.
If $d_p(u_n, u) \to 0$, then $u\in S^{m,p}$
and
$\limsup_{n\to \infty}\mbox{ker}\;  u_n  \subset  \mbox{ker}\;  u$.
\end{tl}

\begin{proof} \ The desired result follows immediately from the proof of Theorem \ref{sgmcde}. \end{proof}

\begin{tl}
  $(S^m, d_p)$ is a closed subspace of  $(\widetilde{S}^m, d_p)$.
\end{tl}

\begin{proof} \ The desired result follows immediately from Corollary \ref{ukc}. \end{proof}

  \begin{tm} \label{sdln}
  $S^m$ is a dense set in $(S^{m,p}, d_p)$.
\end{tm}

\begin{proof} \ The proof is similar to the proof of Theorem \ref{sln}. \end{proof}

\begin{tm} \label{vmce}
   $(E^{m,p}, d_p)$   is a closed subspace of $(S^{m,p}, d_p)$.
\end{tm}

\begin{proof} \ By Proposition \ref{kcs}, we know that $( K_C(\mathbb{R}^m), H )$
is
a closed set in $( K_S(\mathbb{R}^m), H )$.
In a way similar to the proof
of
Theorem \ref{sgmcde},
we
can
obtain
the desired result by using this fact. \end{proof}

\begin{tl}
   $(E^m, d_p)$ is a closed subspace of $(S^m, d_p)$.
\end{tl}

\begin{proof} \ The desired results follows immediately from Theorem \ref{vmce}. \end{proof}

\begin{tm} \label{edln}
  $E^m$ is a dense set in $(E^{m,p}, d_p)$.
\end{tm}

\begin{proof} \ The proof is similar to the proof of Theorem \ref{sdln}. \end{proof}

\begin{tm} \label{gsc}
 $ (\widetilde{S}^{m,p}, d_p)$ is the completion of  $ (\widetilde{S}^{m}, d_p)$.
\end{tm}

\begin{proof} \ The desired result follows from Theorems \ref{scp}, \ref{gsmc}, \ref{slne}. \end{proof}

\begin{tm} \label{smc}
  $(S^{m,p}, d_p)$    is the completion of  $(S^m, d_p)$.
\end{tm}

\begin{proof} \ The desired result follows from Theorems \ref{sgmcde}, \ref{sdln}, \ref{gsc}. \end{proof}

\begin{tm}\label{evmc}
  $(E^{m,p}, d_p)$    is the completion of  $(E^m, d_p)$.
\end{tm}

\begin{proof} \ The desired result follows from Theorems \ref{vmce}, \ref{edln}, \ref{smc}.     \end{proof}

From Corollary \ref{ukc} and Remark \ref{undg}, we
can
obtain the following two theorems.
\begin{tm}
  \label{sm0plc}
 $(S_0^{m,p}, d_p)$    is the completion of  $(S_0^m, d_p)$.
\end{tm}

\begin{tm}
  \label{sm0c}
 $(S^m_0, d_p)$ is a closed subspace of $(S^m, d_p)$.
\end{tm}

%
\begin{figure}
  \centering
  \includegraphics{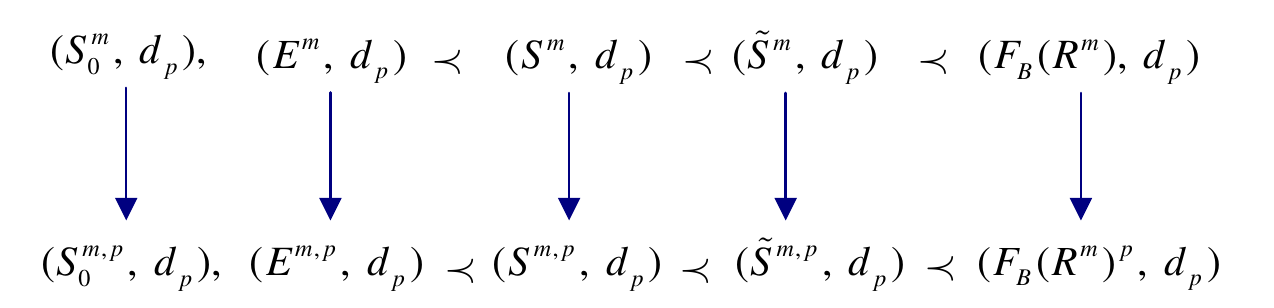}\\
  \caption{Relationship among various subspaces of $(F(\mathbb{R}^m)^p, d_p)$.\newline
 $A\prec B$ denotes that $A$ is a closed subspace of $B$ and
$A\longrightarrow B$ means that $B$ is the completion of $A$. }\label{sbpren}
\end{figure}
 Figure \ref{sbpren}  summarizes all the above results in this section.

 Kr\"{a}tschmer \cite{kratschmer} presented the completion of $(E^m, d_p)$
which is generated and described via
the support functions of fuzzy numbers.
In this paper,
we find that
 the completions of $(E^m, d_p)$, $(S_0^m, d_p)$, $(S^m, d_p)$, $(\widetilde{S}^m, d_p)$ and $(F_{B}   (  \mathbb{R}^m ), d_p)$
can be obtained
by means
of
$L_p$-extension.
To a certain extent, 
completions illustrated in this form are   more concise and clear and
so
more
 easy to perform theoretical research and practical application.

Now, based
on
these
 statements
about relationship among
 subspaces
of $(F(\mathbb{R}^m)^p, d_p)$,
and characterizations
of
compactness given in Sections 4 and 5,
we
intend
to give
characterizations
of totally bounded sets, relatively compact sets and compact sets
in
subspaces of $(F(\mathbb{R}^m)^p, d_p)$.

\begin{tm} \label{spcn}
Let
  $U \subset \widetilde{S}^{m,p}$ ($ U \subset E^{m,p}$, $ U \subset  {S}_0^{m,p}$, $  U \subset S^{m,p}$),
then $U$ is  totally bounded  if and only if it is relatively compact  in $ (\widetilde{S}^{m,p},   d_p)$ ($(E^{m,p}, d_p)$, $( {S}_0^{m,p}, d_p)$, $( S^{m,p}, d_p)$),
which is equivalent to
 \\
 (\romannumeral1) \  $U$ is uniformly $p$-mean bounded, and
 \\
(\romannumeral2) \ $U$ is $p$-mean equi-left-continuous.
\end{tm}

\begin{proof} \ Note that in a complete space, a set is totally bounded if and only if it is relatively compact.
 So
the desired results follow from Theorems \ref{pcn}
and the completeness
of
$ (\widetilde{S}^{m,p},   d_p)$, $(E^{m,p}, d_p)$, $( {S}_0^{m,p}, d_p)$ and $( S^{m,p}, d_p)$. \end{proof}

Let
$U \subset \widetilde{S}^{m,p}$ ($U \subset E^{m,p}$, $U \subset {S}_0^{m,p}$, $U \subset S^{m,p}$).
It
is easy to see that
$\overline{U}$
is exactly
the closure of $U$
in $\widetilde{S}^{m,p}$ ($E^{m,p}$, ${S}_0^{m,p}$, $ S^{m,p}$).

\begin{tm} \label{smpcn} Let
  $U$ be a set in  $ (\widetilde{S}^{m,p},   d_p)$  ($(E^{m,p}, d_p)$, $( {S}_0^{m,p}, d_p)$, $( S^{m,p}, d_p)$),
then
$U$ is    compact  in
$ (\widetilde{S}^{m,p},   d_p)$ ($(E^{m,p}, d_p)$, $( {S}_0^{m,p}, d_p)$, $( S^{m,p}, d_p)$)
 if and only if
 \\
 (\romannumeral1) \  $U$ is uniformly $p$-mean bounded;
 \\
(\romannumeral2) \ $U$ is $p$-mean equi-left-continuous;
\\
 (\romannumeral3) \  $U= \overline{U}$.
\end{tm}

\begin{proof} \ The desired results follow immediately from Theorem \ref{spcn}. \end{proof}

\begin{tm} Let
  $U \subset \widetilde{S}^{m}$ ($U \subset E^m$, $U \subset {S}_0^{m}$, $U \subset S^m$),
then $U$ is     totally bounded  in $ (\widetilde{S}^{m},   d_p)$ ($(E^{m},   d_p)$, $ (S_0^{m},   d_p)$, $ (S^{m},   d_p)$)
 if and only if
 \\
 (\romannumeral1) \  $U$ is uniformly $p$-mean bounded;
 \\
(\romannumeral2) \ $U$ is $p$-mean equi-left-continuous.
\end{tm}

\begin{proof} \ The desired conclusion follows immediately
from
Theorem \ref{tcn}. \end{proof}

\begin{tm} Let
  $U \subset \widetilde{S}^{m}$ ($U \subset E^m$, $U \subset {S}_0^{m}$, $U \subset S^m$),
then $U$ is    relatively compact in $ (\widetilde{S}^{m},   d_p)$ ($(E^{m},   d_p)$, $ (S_0^{m},   d_p)$, $ (S^{m},   d_p)$)
 if and only if
 \\
 (\romannumeral1) \  $U$ is uniformly $p$-mean bounded;
 \\
(\romannumeral2) \ $U$ is $p$-mean equi-left-continuous;
\\
 (\romannumeral3) \  $\overline{U} \subset \widetilde{S}^{m}$
($\overline{U} \subset E^{m}$,   $\overline{U} \subset {S}_0^{m}$, $\overline{U} \subset S^m$).
\end{tm}

\begin{proof} \ The desired conclusion follows immediately
from
Theorem \ref{spcn}. \end{proof}

\begin{tm} \label{gsmchra} Let
  $U \subset \widetilde{S}^{m}$ ($U \subset E^m$, $U \subset {S}_0^{m}$, $U \subset S^m$),
then $U$ is   compact in $ (\widetilde{S}^{m},   d_p)$ ($(E^{m}, d_p)$, $( {S}_0^{m}, d_p)$, $( S^{m}, d_p)$)
 if and only if
 \\
 (\romannumeral1) \  $U$ is uniformly $p$-mean bounded;
 \\
(\romannumeral2) \ $U$ is $p$-mean equi-left-continuous;
\\
(\romannumeral3) \  $\overline{U} = U$.
\end{tm}

\begin{proof} \ The desired results follow from Theorem \ref{gscn} and the completeness
of
$ (\widetilde{S}^{m,p},   d_p)$, $(E^{m,p}, d_p)$, $( {S}_0^{m,p}, d_p)$ and $( S^{m,p}, d_p)$. \end{proof}

\begin{tm} \label{fbsrchrae} Let
  $U \subset \widetilde{S}^{m}$ ($U \subset E^m$, $U \subset {S}_0^{m}$, $U \subset S^m$),
then $U$ is relatively   compact in $(\widetilde{S}^{m}, d_p)$ ($( E^m, d_p)$, $({S}_0^{m}, d_p)$, $(S^m, d_p)$)
 if and only if
 \\
 (\romannumeral1) \  $U$ is uniformly $p$-mean bounded;
 \\
(\romannumeral2) \ $U$ is $p$-mean equi-left-continuous;
\\
(\romannumeral3$'$) \
Given $\{u_n: n=1,2,\ldots\} \subset U$, there exists a $r>0$ and a subsequence $\{v_n\}$ of $\{u_n\}$ such
that
$\lim_{n\to \infty}    |v_n|^r =0$.

\end{tm}

\begin{proof} \ Note that $F_B(\mathbb{R}^m) \cap \widetilde{S}^{m,p} = \widetilde{S}^{m}$,
$F_B(\mathbb{R}^m) \cap {E}^{m,p} = {E}^{m}$,
$F_B(\mathbb{R}^m) \cap {S}_0^{m,p} = {S}_0^{m}$,
$F_B(\mathbb{R}^m) \cap {S}^{m,p} = {S}^{m}$,
so
we can obtain
the desired results by applying Theorems \ref{fbsrchra} and \ref{spcn}. \end{proof}

\begin{tm} \label{fbschrae} Let
  $U$ be a set in  $\widetilde{S}^{m}$ ($ E^m$, ${S}_0^{m}$, $S^m$),
then $U$ is compact in
 $(\widetilde{S}^{m}, d_p)$ ($( E^m, d_p)$, $({S}_0^{m}, d_p)$, $(S^m, d_p)$)
 if and only if $U$ is closed
in
$(\widetilde{S}^{m}, d_p)$ ($( E^m, d_p)$, $({S}_0^{m}, d_p)$, $(S^m, d_p)$) and satisfies
conditions (\romannumeral1), (\romannumeral2)
and
 (\romannumeral3$'$)
 in Theorem \ref{fbsrchrae}.
\end{tm}

\begin{proof} \ The desired result follows immediately from Theorem \ref{fbsrchrae}. \end{proof}

\section{Applications of the results in this paper}
%
%
%

In this section, by using results in this paper,
we
consider some properties of $d_p$ convergence on fuzzy sets spaces.
Then
we
relook   the characterizations of compactness
presented
in
\cite{wu2,zhao}.



Now we discuss properties of $d_p$ convergence from the point of view of level sets of fuzzy sets.

Let
$\{u_k,\ k=1,2,\ldots\} \subset F_{GB}(\mathbb{R}^m)$ and $u\in F_{GB}(\mathbb{R}^m)$.
For writing convenience,
the symbol
$H([u_k]_\alpha, [u]_\alpha) \st{   \mbox{a.e.}   }{\to}   0 \; (P)$
is used
to
denote
that
$H([u_k]_\alpha, [u]_\alpha) \to 0$ as $k\to \infty$ holds almost everywhere on $\alpha\in P$.

\begin{lm}\label{dpce}
  Suppose that $\{u_k,\ k=1,2,\ldots\} \subset F_{GB}(\mathbb{R}^m)$, $u\in F_{GB}(\mathbb{R}^m)$ and  $\al\in (0,1]$.
Then
$d_p(u_k^{(\al)}, u    ) \to 0$
is equivalent to
\\
  (\romannumeral1) \ $H([u_k]_\beta, [u]_\beta) \st{   \mbox{a.e.}   }{\to} 0 \; ([\al,1])$.
\\
 (\romannumeral2) \  $H([u_k]_\alpha, [u]_\alpha) \to 0$ and $[u]_\gamma \equiv [u]_\al$ when $\gamma \in [0, \al]$.
\end{lm}

\begin{proof} \ The desired conclusion follows from the definition of $d_p$ metric and $u_k^{(\al)}$. \end{proof}

\begin{re}
  {\rm Take $\al\in (0,1]$. From Lemma \ref{dpce},
we
know
that
even if $\{u_k\}$ is a convergent sequence in $(F_{B}   (  \mathbb{R}^m )^p,  d_p)$,
$\{u_k^{(\al)}\}$
is not necessarily
 a convergent sequence in $(F_{B}   (  \mathbb{R}^m )^p,  d_p)$.
The following example is given
to
show
this fact.
This example is a small change of Example 4.1 in \cite{wu2}.
}
\end{re}

\begin{eap}
{\rm
  Consider a sequence $\{ u_n: n=1,2,\ldots   \}$ in $E^1$ defined by
\begin{align*}
      & u_n(x)=\left\{
         \begin{array}{ll}
           1-2x, & x\in [0, \frac{1 }{3}], \\
         \frac{1 }{3} -   \frac{3}{2(n+3)}\left(x-  \frac{1 }{3} \right)     , & x\in [\frac{1 }{3},  1], \\
           0, &  x \notin [0,1],
         \end{array}
       \right.
\ n=1,3,5,\ldots,
\\
 & u_n(x)=\left\{
         \begin{array}{ll}
           1-2x, & x\in [0, \frac{1 }{3}], \\
     \frac{1 }{3}, & x\in [ \frac{1 }{3},  \frac{2 }{3} ],\\
         \frac{1 }{3} -   \frac{3}{(n+3)}\left(x-  \frac{2 }{3} \right)     , & x\in [\frac{2 }{3},  1], \\
           0, &  x \notin [0,1],
         \end{array}
       \right.
\ n=2,4,6, \ldots.
\end{align*}
It can be checked that $d_p(u_n, u_0) \to 0$, where
\begin{align*}
   u_0(x)=\left\{
         \begin{array}{ll}
           1-2x, & x\in [0, \frac{1 }{3}], \\
         \frac{1 }{3}    , & x\in [\frac{1 }{3},  1], \\
           0, &  x \notin [0,1].
         \end{array}
       \right.
\end{align*}
So we know
that
$\{ u_n  \}$ is a convergent sequence in $(E^1, d_p)$.
But
$\{   u_n^{(\frac{1}{3})}   \}$
is
a
divergent sequence in $(E^1, d_p)$.
In fact,
we can see that
$d_p(u_{2n-1}^{(\frac{1}{3})}, v) \to 0$
and
$d_p(u_{2n}^{(\frac{1}{3})}, w) \to 0$,
where
\begin{align*}
   & v(x)=\left\{
         \begin{array}{ll}
           1-2x, & x\in [0, \frac{1 }{3}], \\
                  0, &  x \notin [0, \frac{1 }{3}],
         \end{array}
       \right.
\\
 & w(x)=\left\{
         \begin{array}{ll}
           1-2x, & x\in [0, \frac{1 }{3}], \\
         \frac{1 }{3}    , & x\in [\frac{1 }{3},   \frac{2 }{3} ], \\
           0, &  x \notin [0,1].
         \end{array}
       \right.
\end{align*}
Clearly, $v \not= w$.
Thus
$\{       u_n^{  (\frac{1}{3})  }       \}$ is not a convergent sequence in $(E^1, d_p)$.

}
\end{eap}

\begin{tm} \label{subscng}
 Suppose that  $U\subset  F_{GB}   (  \mathbb{R}^m )$ satisfying
that
 $\{[u]_\alpha: u\in U\}$
   is a bounded set in $K(\mathbb{R}^m)$ for each $\al \in (0,1]$.
Let
 $\{r_i, i=1,2,\ldots\}$ be a sequence in $(0,1]$.
Then
given $\{u_n\} \subset U$,
it has a subsequence
$\{ u_{n_k} \} $
satisfying the following two statements.
\\
 (\romannumeral1) \ There exists a $u_0\in  F_{GB}   (  \mathbb{R}^m )$
such that
$H([u_{n_k}]_\alpha, [u_0]_\alpha) \st{   \mbox{a.e.}   }{\to}   0 \; ([0,1])$.
\\
 (\romannumeral2) \ There exist $u_{r_i}, i=1,2,\ldots$, in $K   (  \mathbb{R}^m )$
such that
$H([u_{n_k}]_{r_i}, u_{r_i}) \to  0 $ for each $r_i$.

Thus
$d_p( u_{n_k}^{(r_i)} , u(r_i) ) \to 0$ for each $r_i$, where $ u(r_i) \in F_{B}   (  \mathbb{R}^m ) $, $i=1,2,\ldots$, is defined by
\[
[u(r_i)]_\al=\left\{
               \begin{array}{ll}
                 [u_0]_\al, & \al \in (r_i,1],  \\
                 u_{r_i}, & \al\in [0, r_i].
               \end{array}
             \right.
\]
\end{tm}

\begin{proof} \  Given $\{u_n\} \subset U$. Since
that
$\{[u]_\alpha: u\in U\}$
   is a bounded set in $K(\mathbb{R}^m)$ for each $\al \in (0,1]$,
then
from the proof of Theorem \ref{pcn},
we know that
$\{u_n\} $
has
a subsequence
$\{ u_{n_j} \} $
such
that
$H([u_{n_j}]_\alpha, [u_0]_\alpha) \st{   \mbox{a.e.}   }{\to}   0 \; ([0,1])$,
where
$u_0$ is    one of the elements
in
$ F_{GB}   (  \mathbb{R}^m )$.

Note that
for each $\alpha>0$,
$\{[u]_\alpha: u\in U\}$ is a relatively compact set in
$(K(\mathbb{R}^m), H)$.
Now, by using the
diagonal method
and
proceed according to the proof of Theorem \ref{pcn},
we
can choose a subsequence $\{u_{n_k}\}$ of $\{ u_{n_j} \} $
such
that
 $H([u_{n_k}]_{r_i}, u_{r_i}) \to  0 $ for each $r_i$,
  where
$\{u_{r_i}, i=1,2,\ldots\}$ is a sequence of elements
in    $ K   (  \mathbb{R}^m )$.

So $\{u_{n_k}\}$
is
a subsequence of $\{u_n\} $which
satisfies
 statements    (\romannumeral1) and (\romannumeral2).
Thus,
by Lemma \ref{dpce}, we know that $d_p( u_{n_k}^{(r_i)} , u(r_i) ) \to 0$ for all $r_i$.
\end{proof}

\begin{re} \label{subc}
{\rm
  Suppose that $U \subset F_{B}   (  \mathbb{R}^m )^p$ satisfies conditions (\romannumeral1)
 in Theorem \ref{pcn}, i.e.,
$U$ is uniformly $p$-mean bounded.
Let
 $\{r_i, i=1,2,\ldots\}$ be a sequence in $(0,1]$.
Then
given $\{u_n\} \subset U$,
there is a subsequence $\{u_{n_k}\}$ of $\{u_n\} $
such
that $\{ u_{n_k}^{(r_i)} \}$ converges
to
$ u(r_i) \in F_{B}   (  \mathbb{R}^m )  $ in $d_p$ metric
for each
$r_i$.

In fact, if $U \subset F_{B}   (  \mathbb{R}^m )^p$ satisfies conditions (\romannumeral1)
 in Theorem \ref{pcn}, then
$\{[u]_\alpha: u\in U\}$
   is a bounded set in $K(\mathbb{R}^m)$ for each $\al \in (0,1]$.
So above conclusion follows
immediately
from Theorem \ref{subscng}.

}
\end{re}

From    Theorem \ref{pcn} and Lemma \ref{dpce},
we
can obtain two corollaries.

\begin{tl}
\label{cccp}
  Suppose that $\{u_k,\ k=1,2,\ldots\} \subset F_{B}   (  \mathbb{R}^m )^p$
  and
  $u\in F_{GB}   (  \mathbb{R}^m )$.
  Then
 statements (\romannumeral1) and (\romannumeral2) listed below are equivalent.
  \\
  (\romannumeral1) \ There exists a decreasing sequence $\{r_i, i=1,2,\ldots\}$ in $(0,1]$
converging to zero such that $d_p( u_k^{(r_i)},   u^{(r_i)}    ) \to    0$ for all $r_i$.
  \\
  (\romannumeral2) \  $H([u_k]_\alpha, [u]_\alpha) \st{   \mbox{a.e.}   }{\to}   0 \ (  [0,1]    )$.

Furthermore, if  $U:=\{u_k,\ k=1,2,\ldots\} $ satisfies conditions (\romannumeral1)
and
(\romannumeral2) in Theorem \ref{pcn},
then above   statements (\romannumeral1) and (\romannumeral2) are equivalent
to
statement (\romannumeral3) below.
\\
  (\romannumeral3) \ $d_p( u_k,   u    ) \to    0$ and $u\in F_{B}   (  \mathbb{R}^m )^p$.

\end{tl}

\begin{tl}\label{crm}
Suppose that $\{u_k, \ k=1,2,\ldots\} \subset F_{B}   (  \mathbb{R}^m )^p$,
that
$
\{r_i, i=1,2,\ldots\}$ is a decreasing sequence in $(0,1]$ which converges to zero,
  and
that
    $u(r_i)\in F_{GB}   (  \mathbb{R}^m )$, $i=1,2,\ldots$.
Then
the following statements are equivalent.
\\
 (\romannumeral1) \ $d_p( u_k^{(r_i)},   u  (r_i)    ) \to    0$ for all $r_i$.
  \\
  (\romannumeral2) \ There exists a unique
 $u_0 \in F_{GB}$
such that
$
H([u_k]_\alpha, [u_0]_\alpha) \st{   \mbox{a.e.}   }{\to}   0 \ (  [0,1]    )
$
and
$[u_0]_\al =   [u_0^{(r_i)}]_\al =    [u(r_i)]_\al$ when $r_i < \alpha \leq 1$.
Moreover,
for each $r_i$, $H([u_k]_{r_i}, [u(r_i)]_{r_i}) \to 0$ and $[u(r_i)]_{r_i} \equiv [u(r_i)]_{\theta}$ when $\theta\in [0, r_i]$.

Furthermore, if  $U:=\{u_k,\ k=1,2,\ldots\} $ satisfies conditions (\romannumeral1)
and
(\romannumeral2) in Theorem \ref{pcn},
then above   statements (\romannumeral1) and (\romannumeral2) are equivalent
to
below
statement (\romannumeral3).
\\
  (\romannumeral3) \ There exists a unique
 $u_0 \in F_{B}   (  \mathbb{R}^m )^p$
such that
$d_p( u_k,   u_0    ) \to    0$
and
$[u_0]_\al =   [u_0^{(r_i)}]_\al =    [u(r_i)]_\al$ when $r_i < \alpha \leq 1$.
Moreover,
for each $r_i$, $H([u_k]_{r_i}, [u(r_i)]_{r_i}) \to 0$ and $[u(r_i)]_{r_i} \equiv [u(r_i)]_{\theta}$ when $\theta\in [0, r_i]$.
\end{tl}

\begin{re}  {\rm
The readers can see \cite{huang7, trutschnig} for more studies on this topic,
which
 consider the relations between $d_p$ convergence and other types of convergence on fuzzy sets spaces.

}
\end{re}

The foregoing results enable us to
relook   the characterizations of compactness proposed in
previous work.

Compare   Theorem \ref{smpcn} with   Proposition \ref{gep}.
We can see
that,
in contrast to the former,
the latter has an additional condition (\romannumeral3):
\begin{description}
 \item[]
Let $r_i$ be a decreasing sequence in $(0,1]$
converging to zero.
For $\{u_k\} \subset U$,   if $\{u_k^{(r_i)} : k=1,2,\ldots \}$
converges to
$u(r_i) \in E^m $ in $d_p$ metric, then there exists a $u_0 \in E^m$
such that
$[u_0^{(r_i)}]_\al=[u(r_i)]_\al$ when $r_i < \alpha \leq 1$.
\end{description}
The
reason is that $(E^m, d_p)$
is not complete.
The function of ``conditions (\romannumeral3)''
in Proposition \ref{gep}
is to guarantee completeness of the closed subspace $(U, d_p)$ of $(E^m, d_p)$.

In fact,
if $U$ satisfies conditions (\romannumeral1) and (\romannumeral2) of
Proposition \ref{gep},
then from Remark \ref{subc},
given
$\{u_k\} \subseteq U$,
there
exists
a subsequence $\{u_{n_k} \}$ of $\{u_k\} $
such
that
 $\{u_{n_k}^{(r_i)} : k=1,2,\ldots \}$
converges to
$u(r_i) \in E^m $ in $d_p$ metric for each $r_i$.
Thus
from Corollary \ref{crm},
$\{  u_{n_k}  \}$
is
a convergent sequence in $(E^{m,p}, d_p)$,
and
its
limit point    $u_0 \in E^{m,p}$ is
defined
by
$$[u_0]_\al = [u_0^{(r_i)}]_\al : =   [u(r_i)]_\al \ \mbox{when} \ r_i < \alpha \leq 1,$$
for each $\al \in (0,1]$.
So
it
is
easy to see
that
if $U$ satisfies conditions (\romannumeral1) and (\romannumeral2) of
Proposition \ref{gep},
then
condition (\romannumeral3) in Proposition \ref{gep}
is equivalent to
$\overline{U}  \subset E^m$.

Thus we know
that
Proposition \ref{gep}
can be written     in the following 
form:

\noindent\textbf{Proposition \ref{gep}$'$} A closed set $U\subset (E^m, d_p)$ is compact if and only if:
\\
(\romannumeral1) $U$ is uniformly $p$-mean bounded;
\\
(\romannumeral2) $U$ is $p$-mean equi-left-continuous;
\\
(\romannumeral3) $\overline{U}  \subset E^m$.


Similarly,
Propositions \ref{gs0p} and \ref{gsp} can be written as:

\noindent\textbf{Proposition \ref{gs0p}$'$} A closed set $U$ in $(S^m_0, d_p)$ is compact if and only if:
\\
(\romannumeral1) $U$ is uniformly $p$-mean bounded;
\\
(\romannumeral2) $U$ is $p$-mean equi-left-continuous;
\\
(\romannumeral3) $\overline{U}  \subset S^m_0$.

\noindent\textbf{Proposition \ref{gsp}$'$} A closed set $U$ in $ (S^m, d_p)$ is compact if and only if:
\\
(\romannumeral1) $U$ is uniformly $p$-mean bounded;
\\
(\romannumeral2) $U$ is $p$-mean equi-left-continuous;
\\
(\romannumeral3) $\overline{U}  \subset S^m$.

\section{Conclusions}

We
give
a characterization of totally boundedness for fuzzy sets space endowed with $d_p$ metric.
This is
the
key result of this paper.
To prove this conclusion,
we introduce
the
$L_p$-extension of a fuzzy sets space.
Then
we present
a characterization of relatively compactness in $(F_{B}(  \mathbb{R}^m )^p, d_p)$ which is the
$L_p$-extension of
$(F_{B}(  \mathbb{R}^m ), d_p)$.
It's
shown
that this characterization is also a    characterization of totally boundedness in $(F_{B}(  \mathbb{R}^m )^p, d_p)$.
From
this fact, we know
that
this
characterization is a characterization of totally boundedness
in
any subspaces of $(F_{B}(  \mathbb{R}^m )^p, d_p)$
because
totally boundedness
is dependent only
on
the set itself and
the
metric $d_p$.
From this fact,
we
can also deduce that $(F_{B}(  \mathbb{R}^m )^p, d_p)$ is the completion of $(F_{B}(  \mathbb{R}^m ), d_p)$.
Based on this,
we
give characterizations of relatively compactness and compactness
in $(F_{B}(  \mathbb{R}^m ), d_p)$, $(F_{B}(  \mathbb{R}^m )^p, d_p)$ and their subspaces.
It is found
that
the completion of
each fuzzy sets space mentioned in this paper
is precisely
its
$L_p$-extensions.
Relationship
among
all the spaces mentioned in this paper are clarified and summarized in a figure.
At last, as applications of results in
this
paper,
we discuss some properties of $d_p$ convergence
and
relook characterizations of compactness
proposed in previous work.

Since the input-output relation of fuzzy systems can be described using fuzzy functions \cite{zeng, huang8, wa},
the
results in this paper can be used to analysis and design
fuzzy systems.
 Compactness criteria is used to solve fuzzy differential equations \cite{roman2},
so this topic is a potential application of our results.

\end{document}